\documentclass[final]{svjour3}

\def\makeheadbox{}

\usepackage{graphicx}
\usepackage{mathptmx}      
\usepackage{latexsym}
\journalname{Preprint}
\usepackage{mathptmx}
\usepackage{amssymb,amsmath,amsfonts}
\usepackage{bm}
\usepackage[english]{babel}
\usepackage{babelbib}
\usepackage[T1]{fontenc}
\usepackage{verbatim}
\usepackage{xcolor}
\usepackage{mathrsfs}
\usepackage{placeins}
\usepackage{longtable}
\usepackage{subfigure}
\usepackage{fancyhdr}
\usepackage{dsfont}
\numberwithin{equation}{section}
\newcommand\ol[1]{\overline{#1}}
\newcommand\nl{\\[0.75\jot]} 
\newcommand\half{\tfrac{1}{2}}
\newcommand\shalf{\frac{1}{2}}
\newcommand{\eps}{\varepsilon}
\newcommand{\sig}{\sigma}
\renewcommand{\t}{\tau}
\renewcommand{\th}{\vartheta}
\newcommand\Laplace{\Delta}
\newcommand{\RR}{{{\mathbb{R}}\vphantom{|}}}
\newcommand{\dd}{{\mathrm{d}}}
\newcommand{\ee}{{\mathrm{e}}}
\newcommand{\ii}{{\mathrm{i}}}

\newcommand{\nice}[1]{\mathcal{#1}} 

\newcommand{\nE}{\nice{E}}

\newcommand{\nG}{\nice{G}}

\newcommand{\nL}{\nice{L}}

\newcommand{\nO}{\nice{O}}
\newcommand{\nP}{\nice{P}}

\newcommand{\nS}{\nice{S}}

\renewcommand\,{\ensuremath{\mspace{1.5mu}}} 
\newcommand\Order{{\mathscr{O}}}
\newcommand{\bcor}[1]{{#1}}
\newcommand{\ccor}[1]{{#1}}

\newcommand{\ncor}[1]{#1}

\newcommand\pdd{\partial_2}
\newcommand\EA{\nE_A}
\newcommand\EB{\nE_B}
\newcommand\EF{\nE_F}

\renewcommand\L[1]{{\nL}^{(#1)}}
\renewcommand\S[1]{\nS^{(#1)}}
\newcommand\s[1]{\bm{s}^{(#1)}}

\newcommand\pdt[1]{\tfrac{\partial^{#1}}{\partial t^{#1}}}
\newcommand\pdtau[1]{\tfrac{\partial^{#1}}{\partial \tau^{#1}}}
\DeclareMathOperator{\Tr}{Tr}
\begin{document}
\title{Adaptive splitting methods for nonlinear Schr\"{o}dinger equations in the semiclassical regime
\thanks{This work was supported by the Austrian Science Fund (FWF) under grants P24157-N13 and P21620-N13, the Vienna Science and Technology Fund (WWTF) under the grant MA14-002 and the \bcor{Doktoratsstipendium of the University of Innsbruck}.}}



\author{Winfried Auzinger, Thomas Kassebacher, Othmar Koch, and Mechthild Thalhammer}


\institute{Winfried Auzinger \at
              Institute for Analysis and Scientific Computing, Vienna University of Technology, Wiedner Hauptstra{\ss}e \mbox{8-10}, A-1040 Wien, Austria \\
              \email{w.auzinger@tuwien.ac.at}           
           \and
           Thomas Kassebacher \at
              Institut f{\"u}r Mathematik, Leopold-Franzens Universit{\"a}t Innsbruck, Technikerstra{\ss}e 13, A-6020 Innsbruck, Austria\\
              \email{thomas.kassebacher@uibk.ac.at}
           \and
           Othmar Koch \at
              Faculty of Mathematics, University of Vienna, Oskar-Morgenstern-Platz~1, A-1090 Wien, Austria \\
              \email{othmar@othmar-koch.org}
           \and
           Mechthild Thalhammer \at
              Institut f{\"u}r Mathematik, Leopold-Franzens Universit{\"a}t Innsbruck, Technikerstra{\ss}e 13, A-6020 Innsbruck, Austria\\
              \email{mechthild.thalhammer@uibk.ac.at}
}

\maketitle

\begin{abstract}
\bcor{
The error behavior of exponential operator splitting methods for nonlinear Schr{\"o}dinger equations in the semiclassical regime is studied.
For the Lie and Strang splitting methods, the exact form of the local error is determined and the dependence on the semiclassical parameter is identified.
This is enabled within a defect-based framework which also suggests asymptotically correct a~posteriori local error estimators as the basis for adaptive time stepsize selection.
Numerical examples substantiate and complement the theoretical investigations.
}
\keywords{Nonlinear Schr{\"o}dinger equations
 \and Semiclassical regime
 \and Splitting methods
 \and Adaptive time integration
 \and Local error
 \and Convergence}
\subclass{65J10 \and 65L05 \and 65M12 \and 65M15}
\end{abstract}

\section{Introduction}

\paragraph{Discretization of semiclassical Schr\"{o}dinger equations.}
\bcor{
The quantitative and qualitative behavior of space and time discretization methods for linear and nonlinear Schr\"{o}dinger equations has been extensively
studied in recent years.
As a small selection, we mention the contributions~\cite{Bader2014,BaoJinMarkowich2002,BaoJinMarkowich2003,Carles2011,Degond2007,Gradinaru2014,Jin2011,Yang2014} which are of relevance in particular in the context of semiclassical Schr\"{o}dinger equations.
}

\bcor{
The numerical approximation of nonlinear Schr\"{o}dinger equations in the semiclassical regime is a challenge, since in general the space
and time increments have to be chosen in dependence of the semiclassical parameter $0 < \eps < 1$ in order to correctly capture the solution behavior.
In particular, for an initial state~$u$ depending on the parameter~$\eps$ in the form of semiclassical wave packets, WKB states or focussing states, the solution~$\psi$ shows a highly oscillating behavior.
However, a precise characterization of the solution to semiclassical nonlinear Schr{\"o}dinger equations in dependence of the prescribed initial state is a question in the area of analysis that has not been resolved exhaustively yet.
}

\paragraph{Operator splitting methods.}
\bcor{
Exponential operator splitting methods for nonlinear Schr\"{o}dinger equations have been in the focus of interest of both theoretical physics and numerical analysis in the last years.
A comprehensive review of numerical methods for nonlinear Schr\"{o}dinger equations such as Gross--Pitaevskii equations is~\cite{baoetal13a}, which summarizes most of the studies conducted in this field.
Time-splitting methods in conjunction with spectral space discretizations are overall concluded to be the most successful approximations, with favorable stability and efficiency as well as norm and energy conservation.
In particular, the spectral accuracy of the space discretization is advantageous for Schr\"{o}dinger equations with regular solutions.
}

\bcor{
A first error analysis of the Lie and Strang splitting methods for nonlinear Schr{\"o}dinger equations is found in~\cite{bess02}.
The seminal work~\cite{lubich07} provides a rigorous convergence analysis for the Strang splitting method applied to the Schr{\"o}dinger--Poisson and cubic Schr{\"o}dinger equations;
extensions to Gross--Pitaevskii equations and high-order splitting methods as well as a study of the effect of spatial discretization by spectral methods are given for instance in~\cite{Gauckler,knth10a,th12}.
The question of long-time integration, with view on near-conservation of invariants under time discretization by splitting methods, is considered in~\cite{faou2012geometric,faou13,gaulub10b}, see also references given therein and~\cite{cangon13,dahowr09} for the analysis of related classes of methods.
}

\paragraph{Objective and outline.}
\bcor{
In this work, we study the cubic Schr{\"o}dinger equation involving a small but fixed parameter $0 < \eps < 1$, see Section~\ref{sec:Problem-setting}.
Our objective is to provide a rigorous a priori and a posteriori local error analysis for low-order splitting methods, the first-order Lie splitting and the second-order Strang splitting methods.
}

\bcor{
First considerations and numerical tests imply that the splitting solution correctly describes the qualitative behavior of the true solution only if the time stepsize $t > 0$ is in the range of the parameter~$\eps$, see also~\cite{BaoJinMarkowich2002,BaoJinMarkowich2003,descombes2013lie}.
Notably, the numerical simulation of the semiclassical limit ($\eps \to 0$) is not possible by the splitting approach.
}

\bcor{
A refined local error analysis provides a deeper understanding of the dependence on the time stepsize and the parameter, see Sections~\ref{sec:locerr-repr}-\ref{Lestim}.
However, as the obtained bounds involve certain Sobolev norms of the solution, whose precise dependence on~$\eps$ is in general unknown, an
appropriate a priori choice of the time stepsize to optimally balance computational cost and accuracy is a delicate issue.
}

\bcor{
Pessimistic bounds for the solution and its spatial derivatives would lead to a systematic underestimation of the time stepsize, at the expense of efficiency.
A remedy is the use of asymptotically correct a posteriori local error estimates for an automatic time step size control, see Section~\ref{sec:apost}.
}

\bcor{
The theoretical investigations are substantiated and complemented by numerical examples, see Section~\ref{sec:numerics}.
}


\bcor{\paragraph{Theoretical results and connection to earlier work.} The present paper extends the work of~\cite{DescombesThalhammer,descombes2013lie}, where the local error in dependence of~$\eps$
is studied for higher-order splitting methods applied to linear equations
and for the first-order Lie splitting method applied to nonlinear equations, respectively.
In particular, we analyze the second-order Strang splitting method in detail, where
we adopt the defect-based approach of~\cite{Auz14,auzingeretal12a,auzingeretal13a}.
This enables us to derive a suitable local error representation with bounds of the form
\begin{align*}
\text{Lie splitting:} &\qquad C\, t^{2}\,  \| u \|_{H^{k}}\,,\\
\text{Strang splitting:}  & \qquad C\, t^{3}\, \eps^{-1}\, \| u \|_{H^{k}}\,.
\end{align*}}
\bcor{Here, the explicitly stated dependence on $\eps$ is associated with the applied splitting method, additionally a solution dependence on $\eps$ may be present. In special cases, precise bounds for $\|u\|_{H^k}$ are known, however, the treatment of the general case is an open analytical question.}
\bcor{Extending~\cite{Auz14}, we analyze asymptotically correct a~posteriori error estimators for the purpose of adaptive time stepping, and verify their asymptotics.}

\section{The \bcor{cubic Schr{\"o}dinger equation} as a model problem}
\label{sec:Problem-setting}
\paragraph{Problem setting.}
The main aim of this paper is
to provide a rigorous a priori and a~posteriori local error analysis
for low-order splitting methods applied to the \bcor{cubic Schr{\"o}dinger equation} (NLS)
\begin{subequations}
\label{eq:Problemcomplete}
\begin{align}
\label{eq:Problem}
\ii\,\pdt{} \psi(t,x) &= -\eps\half\,\Laplace \psi(t,x)
            + \tfrac{1}{\eps}\,\big( U(x)+\th\,|\psi(t,x)|^2 \big)\,\psi(t,x)\,, \\
\psi(0,x)&= u(x)\,,
\end{align}
\end{subequations}
with solution $\psi: [0,T]\times \mathbb{R}^d \rightarrow \mathbb{C}$,
initial state $u:\mathbb{R}^d \rightarrow \mathbb{C}$, a quadratic harmonic potential
\begin{equation}
\label{eq:harmPot}
U: \ \mathbb{R}^d \rightarrow \mathbb{R}:~~ x  \mapsto \half\,\omega^2\,|x|^2\,,
\end{equation}
{\bcor and} a fixed positive constant $0<\eps<1$. \bcor{We choose $\th=1$ to obtain a defocussing nonlinearity, where a solution exists globally}. We focus on the relevant cases $d\in \{1,2,3\}$ and employ certain regularity conditions and boundedness assumptions on $\psi$ and $U\psi$.

\paragraph{Splitting.}
For discretization in time
we split the right-hand side of the PDE~\eqref{eq:Problem},
\begin{align*}
F(\psi)& = \,\ii\,\eps \half\,\Laplace\psi -\ii\,\tfrac{1}{\eps}(U+\th\,|\psi|^2)\,\psi\,,
\end{align*}
separating the two scalings with respect to $\eps$ into
\begin{subequations}
\label{eq:splittingop}
\begin{align}
\label{eq:ProblemA}
A(\psi)=&\,\ii\,\eps \half\,\Laplace\psi\,, \\
\label{eq:ProblemB}
B(\psi)=&-\ii\,\tfrac{1}{\eps}\big( U+\th\,|\psi|^2 \big)\,\psi\,.
\end{align}
\end{subequations}
The evolutionary operators $\EA$ and $\EB$ associated with these subproblems and initial state $u$ are given by
\begin{subequations}
\label{eq:E-def}
\begin{align}
\EA(t,u) &= \ee^{\,\ii\,\,t \frac{\eps}{2}\,\Laplace}\,u\,, \label{eq:EA-def} \\
\EB(t,u) &= \ee^{-\ii\,\,t\frac{1}{\eps}(U+\th\,|u|^2)}\,u\,. \label{eq:EB-def}
\end{align}
\end{subequations}
Representation~\eqref{eq:EA-def} follows from Stone's theorem, and the explicit representation~\eqref{eq:EB-def} is immediate.

For the numerical approximation we consider $s$\,-\,fold splitting
methods, where one splitting step $\nS$ has the general form\footnote{For notational convenience,
                                                                             the time increment is simply denoted by $ t $.}
\begin{equation} \label{eq:splitting-s-stages}
\nS(t,\cdot) := \EB(b_s\,t,\EA(a_s\,t,\ldots ,\EB(b_1\,t,\EA(a_1\,t,\cdot))))\,.
\end{equation}
The splitting coefficients $a_i,\,b_i\in \RR$ are defined by appropriate order conditions.
The numerical solution $\psi_n$ after $n$ time steps is given by
\begin{equation}
\psi_n = \underbrace{\nS(t, \nS(t, \ldots  ,\nS(t, }_{n \text{ times}} u)))\,.
\end{equation}
For the subsequent study the two-fold symmetric second-order Strang splitting method,
with ${a_1=a_2=\frac{1}{2}}$, $b_1=1$, $b_2=0$, see~\eqref{eq:Strang} below,
will be in the focus of interest.

\paragraph{Notation.}
For a nonlinear operator $G$ we denote by $G'$ its Fr{\'e}chet derivative.
Moreover, for the associated evolutionary operator $\nE_G(t,u)$,
its first and $ k $-th derivatives with respect to the initial value
$u$ are denoted by $\pdd \nE_G(t,u)$ and $\pdd^k \nE_G(t,u)$, respectively%
\footnote{The operator $ B $ from~\eqref{eq:ProblemB} is not complex Fr{\'e}chet differentiable due to the occurrence of the factor $ |\Psi|^2 $.
However, this is a merely formal problem, see the discussion in~\cite[Section~5.1]{Auz14}.}.

\section{General representation of the local error} \label{sec:locerr-repr}

Our main goal is a \bcor{better understanding of the local error behavior} of a Strang splitting step 
for problem~\eqref{eq:Problemcomplete}, see Section~\ref{sec:GPE}.
To this end, in the present section we first recapitulate an exact representation
of the local error for a general nonlinear evolution equation
\begin{equation} \label{eq:FAB}
\pdt{} \psi = F(\psi) = A(\psi) + B(\psi)\,,\quad \psi(0) = u\,.
\end{equation}
This local error representation is based on~\cite[Section~4 and Appendix~C]{Auz14}.
Here we do not repeat all details of the derivation
but particularize for $A$ linear as is the case in~\eqref{eq:ProblemA},
and rearrange terms appropriately as a preparation for the subsequent estimates.

Since $A$ is a linear operator, we have
\begin{equation*}
A(u) = A\,u\,,\quad A'(u)\,v = A\,v\,,\quad \EA(t,u) = \EA(t)\,u\,,\quad \pdd\,\EA(t,u)\,v = \EA(t)\,v\,,
\end{equation*}
with the operator exponential $ \EA(t) = \ee^{\,\ii\,\,t \frac{\eps}{2}\,\Laplace} $, see~\eqref{eq:EA-def}.
A Strang splitting step takes the form
\begin{equation} \label{eq:Strang}
\nS(t,u) = \nS_{\text{Strang}}(t,u) = \EA(\half\,t)\,\EB\big(t,\EA(\half\,t)\,u) \big)
\end{equation}
with $ \EB(t,\cdot) $ from~\eqref{eq:EB-def}.

The flow defined by~\eqref{eq:FAB} is denoted by $\EF(t,u)$, and the local error of a splitting step is denoted by
\begin{equation} \label{eq:local-error-general}
\nL(t,u) =  \mathcal{S}(t,u) -\nE_F(t,u)\,.
\end{equation}
The representation of $ \nL(t,u) $ given in the sequel
indicates the expected local order $\Order(t^{p+1})$ of the Strang splitting scheme with $p=2$ \bcor{and in particular the dependence on the operators $A$ and $B$.}
This also will enable us to study the dependence on the parameter~$\eps$.

The approach adopted in~\cite{Auz14}
is based on an iterated application of (non)linear variation of constant formulas
involving the defect $ \S1(t,u) $ of the numerical solution $ \nS(t,u) $, defined according to
\begin{equation} \label{eq:defect-S1}
\pdt{} \nS(t,u) = F(\nS(t,u)) + \S1(t,u)\,,
\end{equation}
such that $ \nS(t,u) $ is the exact solution of the perturbed problem~\eqref{eq:defect-S1}.

\subsection{First expansion step}
Using Proposition~\ref{Grobner} (Gr{\"o}bner-Alekseev formula, see Appendix~\ref{Technical Tools}),
with $ z(t) = \nS(t,u) $, the local error~\eqref{eq:local-error-general} can be written as
\begin{equation} \label{eq:local-error-integral}
\nL(t,u) = \int_{0}^{t} \pdd\,\nE_F(t-\t,\nS(\t,u))\,\S1(\t,u)\,\dd\t\,. 
\end{equation}
An expression for the defect $ \S1(t,u) $ which contains no time derivatives is derived
in Section~\ref{sec:S1herleitung},
\begin{align}
\label{eq:S1expr}
\S1(t,u) &= \EA(\half\,t)\,B(w) - B(\EA(\half\,t)\,w)  \\
& \quad {} + \half\,\EA(\half\,t)\big(\pdd\,\EB(t,v)\,A\,v - A\,\EB(t,v) \big)
               \,\big|_{\,\begin{subarray}~~v = \EA(\shalf t,u)  \\
                  ~w = \EB(t,\EA(\shalf t,u)) \end{subarray}}\,. \notag
\end{align}
Obviously, $\S1(0,u)=0$, hence $\S1(t,u)$ is at least of order $\Order(t)$ provided all
expressions involved are bounded. However, this does not yet reveal the expected
order~$ \Order(t^2) $.

\subsection{Second expansion step}
Further expansion of $\nL(t,u)$ via another application of the variation of constant formula results in
\begin{align} \label{eq:L-double-integral}
\begin{aligned}
\nL(t,u) &=
\int_{0}^{t} \int_{0}^{\t_1}
\Big\{  \pdd\,\EF(t-\t_2,\nS(\t_2,u))\,\S2(\t_2,u) \\
&\qquad\qquad\quad {} + \pdd^2\,\EF(t-\t_2,\nS(\t_2,u))\,
                                  \big(\S1(\t_2,u),\S1(\t_2,u) \big)\Big\}\,\dd\t_2\,\dd\t_1 \\
&=: \L{2}(t,u) + \L{1,1}(t,u) \,,
\end{aligned}
\end{align}
see~\cite{auzinger2014local}, involving the first- and second-order defect terms
\begin{subequations} \label{eq:S12def}
\begin{align}
\S1(t,u) &= \pdt{} \nS(t,u) - F(\nS(t,u))
            \quad \text{(see~\eqref{eq:defect-S1},\,\eqref{eq:S1expr})}\,,\label{eq:S1def} \\
\S2(t,u) &= \pdt\,\S1(t,u) - F'(\nS(t,u))\,\S1(t,u)\,. \label{eq:S2def}
\end{align}
\end{subequations}
Note that for~\eqref{eq:Problem},
\begin{equation*}
F'(u)v = -\eps\half\,\Laplace v + \tfrac{1}{\eps}\,\big(U\,v+\,u^2\,\ol{v}
              +2\,\th\, |u|^2\,v\big)\,.
\end{equation*}%
Again, $ \S2(t,u) $ can be expressed in a way such that no time derivatives occur.
Here we resort to a reformulation facilitating identification of the dominant terms, see Section~\ref{sec:S2herleitung},
\begin{align}
\label{eq:S2expr}
\begin{aligned}
\S2(t,u)
&= \big(\EA(\half\,t)\,B'(w) - B'(\EA(\half\,t)\,w)\,\EA(\half\,t) \big)\,\pdd\,\EB(t,v)\,A\,v  \\
& \quad {} + \big(A + B'(\EA(\half\,t)\,w)\big)\,\big(B(\EA(\half t)\,w) - \EA(\half t)\,B(w) \big)  \\
& \quad {} + \big(\EA(\half\,t)\,B'(w) - B'(\EA(\half\,t)\,w)\,\EA(\half\,t) \big)\,B(w) \\
& \quad {} + \tfrac{1}{4}\,\EA(\half\,t)
            \Big(A \big(A\,\EB(t,v) - \pdd\,\EB(t,v)\,A\,v \big)         \\
& \qquad\qquad\qquad{} - \big(A\,\pdd\,\EB(t,v) - \pdd\,\EB(t,v)\,A \big)\,A\,v  \\
& \qquad\qquad\qquad{} + \pdd^2\,\EB(t,v)(A\,v,A\,v)  \Big) \,
         ~\Big|_{\,\begin{subarray}~~v = \EA(\shalf t)\,u  \\
                    ~w = \EB(t,\EA(\shalf t)\,u) \end{subarray}}\,,
\end{aligned}
\end{align}
satisfying $ \S2(0,u)=0 $, hence $ \S2(t,u) = \Order(t) $
provided all expressions involved remain bounded. Thus, together with $ \S1(t,u) = \Order(t) $,
we have
\begin{align}
\nL(t,u) = \Order(t^3)\,.
\end{align}
Detailed integral expressions for $ \S1(t,u) $ and $\S2(t,u) $ are given
in Section~\ref{subsec:furtherexp} below. 

\subsection{Commutator expressions occurring in the subsequent analysis} \label{subsec:LinearA}
In the expansion of the local error, nonlinear commutators occur.
The commutator of two nonlinear vector fields $A,B$ is defined as\footnote{Here and in the following,
$u$ is a formal variable representing the argument of the respective operators.
This is not to be confused with the initial value, which has also been denoted by $u$.}
\begin{equation*}
[A,B](u) = A'(u)\,B(u) - B'(u)\,A(u)\,.
\end{equation*}
For a linear operator $A$, the relevant first- and second-order commutators are given by
\begin{subequations} \label{eq:ABcommutators}
\begin{align}
[A,B](u) &= A\,B(u)-B'(u)\,A\,u\,, \nl
[A,B'(v)](u) &= A\,B'(v)\,u - B'(v)\,A\,u\,, \nl
[B,[B,A]](u) &= B'(u)\big(B'(u)\,A\,u - A\,B(u) \big) -\big(B'(u)\,A\,u - A\,B(u) \big)'\,B(u)\\
	&= -2\,B'(u)\,A\,B(u) - B''(u)(A\,u,B(u)) + B'(u)\,B'(u)\,A\,u + A\,B'(u)\,B(u)\,, \notag \nl
[A,[A,B]](u)&= A\,\big(A\,B(u) -B'(u)\,A\,u \big) -\big(A\,B(u)-B'(u)\,A\,u \big)'\,A\,u \\
	&= A^2\,B(u) + B''(u)(A\,u,A\,u) + B'(u)\,A^2\,u - 2\,A\,B'(u)\,A\,u\,. \notag
\end{align}
\end{subequations}

\subsection{Explicit integral representation of the local error} \label{subsec:furtherexp}
The integrand in~\eqref{eq:L-double-integral} depends on the first- and second-order defect terms
$ \S1 $ and $ \S2 $, see~\eqref{eq:S1expr},\,\eqref{eq:S2expr}.
For a precise estimation of the local error, a more explicit representation of $ \S1 $ and $ \S2 $
is required. This is accomplished by converting them into integral form according to the
ideas from~\cite{Auz14}. The following integral representations are derived
in~Appendix~\ref{lab:expansiondet}.
\begin{align} \label{eq:S1integral}
\S1(t,u) &= \int_{0}^{t}
\Big\{ \half\,\EA(\half(t-\t))\,[A,B]\,\EA(\half \t)\,w \\
&\qquad\quad {} + \half\,\EA(\half\,t)\,\pdd\,\EB(t-\t,\EB(\t,v))\,[B,A](\EB(\t,v)) \Big\}\,\dd\t\,
                    \;\Big|_{\,\begin{subarray} ~~v =\EA(\half\,t)\,u \notag \\
                    ~w = \EB(t,\EA(\half\,t)\,u) \end{subarray}} \\
&=: \int_{0}^{t} \s1(t,\t,u)\,\dd\t = \Order(t)\,,\notag
\end{align}
and
\begin{subequations}
\label{eq:S2integral}
\begin{align}
\label{eq:S2integral1}
&\S2(t,u) = {} \\
&= \int_{0}^{t}
   \Big\{ -\half\,\EA(\half(t-\t))\,B''(\EA(\half\,\t)\,w) \notag \\
& \quad\qquad\quad {} \cdot \Big(A\,\EA(\half\,\t)\,w,\,\EA(\half\,\t) \cdot
                        \big(\pdd\,\EB(t,v)\,A\,v - A\,\EB(t,v) \big) \Big) \notag \nl
& \quad\qquad\, {} + \half\EA(\half(t-\t)) \notag \\
& \quad\qquad\quad {} \cdot \Big([A,[B,A]](\EA(\half\,\t)\,w) +
[A,B'(\EA(\half\,\t)\,w)]\,\EA(\half\,\t)
                          \big(\pdd\,\EB(t,v)\,A\,v -A\,\EB(t,v) \big) \Big) \notag \\
& \quad\qquad {} + \EA(\half (t-\t)) \notag \\
& \quad\qquad\quad {} \cdot \Big(-B''(\EA(\half\,\t)\,w)\big(\big(\EA(\half\,\t)\,B(w)-B(\EA(\half\,\t)\,w) \big),\,A\,\EA(\half\,\t)\,w \big) \notag \\
& \quad\qquad\quad {} + \half\,[B,[B,A]]\big(\EA(\half\,\t)\,w \big)
                  +  [A,B'(\EA(\half\,\t)\,w)] \big(\EA(\t) B(w)- B(\EA(\half\,\t)\,w)  \big)  \Big) \notag \\
& \quad\qquad {} + \EA(\half\,t)\,\pdd\,\EB(t-\t,\EB(\t,v)) \notag \\
& \quad\qquad\quad {} \cdot \Big(\tfrac{1}{4}\,[A,[A,B]](\EB(\t,v))- \half\,[A,B'(\EB(\t,v))]\big(\pdd\,\EB(\t,v)\,A\,v -A\,\EB(\t,v)\big) \notag \\
& \quad\qquad\quad {} + \tfrac{1}{4}\,B''(\EB(\t,v))\big(\big(\pdd\,\EB(\t,v)\,A\,v
                         - A(\EB(\t,v))\big),\pdd\,\EB(\t,v)\,A\,v + A\,\EB(\t,v) \big) \Big) \Big\} \notag \\
& \qquad\quad \dd\t\;\Big|_{\,\begin{subarray}~~v =\EA(\half\,t)\,u  \\
                  ~w = \EB(t,\EA(\half\,t)\,u) \end{subarray}}\,. \notag
\end{align}
Analogously to~\eqref{eq:S1integral}
we define $\s2(t,\t,u)$ as the integrand in~\eqref{eq:S2integral1}, such that
\begin{equation}
\label{eq:S2integral2}
\S2(t,u) = \int_0^t \s2(t,\t,u)\,\dd\t\,.
\end{equation}
\end{subequations}
Several terms in~\eqref{eq:S2integral} cancel out at $t=0$,
e.g., $ \pdd\,\EB(t,v)\,A\,v - A\,\EB(t,v)$. Hence $\s2$ can be written as
\begin{equation}
\label{eq:s2dominant}
\s2(t,\t,u) = \tfrac{1}{4}\,[A,[B,A]](u) + \half\,[B,[B,A]](u)  + \Order(t)\,.
\end{equation}
These commutators will dominate the term $\S2$.

Combining these results we finally obtain an integral expression for $\nL(t,u)$
consisting of two parts according to~\eqref{eq:L-double-integral}.
Evaluating $ \nL(t,u) $ at $ t=0 $ reveals the dominant term in its Taylor expansion,
\begin{equation}
\label{eq:Strang err-dominant}
\nL(t,u) = \tfrac{t^3}{6}\, \s2(t,\t,u) + \Order(t^4)\,,
\end{equation}
with $\s2$ as in~\eqref{eq:s2dominant}.

\section{The local error for the cubic Schr{\"o}dinger equation} \label{sec:GPE}
In the special case of the NLS~\eqref{eq:Problemcomplete}
the operators $A,B$ are given explicitly in~\eqref{eq:splittingop},
for which we can explicitly calculate the terms appearing in the local error representation.

For $ A\,u = \ii\,\eps \half\,\Laplace u $ from~\eqref{eq:ProblemA} we have
\begin{equation*}
A'(u)\,v \equiv A\,v = \ii\,\eps\,\half\,\Laplace v\,,\quad A''(u)(v,w) \equiv 0\,.
\end{equation*}

\subsection{Auxiliary results for the nonlinear operator $B$} \label{sec:Baux}
In a subsequent $L^2$\,-\,estimate for the integral representation
of the local error, several derivatives of the nonlinear operator from~\eqref{eq:ProblemB},
\begin{equation*}
B(u)= -\ii\,\tfrac{1}{\eps}\,\big( U+\th\,|u|^2 \big)\,u
\end{equation*}
appear.

\paragraph{Fr{\'e}chet derivatives of \,$B$.}
Direct computation yields
\begin{align*}
B'(u)\,v &= -\ii\,\tfrac{1}{\eps}\,\big(U\,v+\th\,(2|u|^2\,v + u^2\,\ol{v}) \big)\,, \\
B''(u)(v,w) &= -2\,\ii\,\tfrac{1}{\eps}\th\,\big(\ol{u}\,v\,w + u\,\ol{v}\,w + u\,v\,\ol{w} \big)\,, \\
B'''(u)(v,w,z) &\equiv -2\,\ii\,\tfrac{1}{\eps}\th\,\big(\ol{v}\,w\,z + v\,\ol{w}\,z + v\,w\,\ol{z} \big)\,.
\end{align*}

\paragraph{Spatial derivatives associated with $B$.}
In the commutators to be analyzed below, spatial derivatives of
the functions $B(u)$, $B'(u)v$ and $B''(u)(v,w)$ occur.
We thus compute
\begin{align*}
\nabla B(u) &= -\ii\,\tfrac{1}{\eps}\,
               \Big((\nabla U)\,u + U (\nabla u)
                     + \th\,\big(2\,|u|^2\,(\nabla u) + u^2\,(\ol{\nabla u}) \big)
               \Big)\,,\nl
\nabla \big(B'(u)\,v\big) & = -\ii\,\tfrac{1}{\eps}\,
               \Big((\nabla U)\,v + U\,\nabla v \\
&\qquad\qquad {}      + \th\,\big(2\,|u|^2\,\nabla v + u^2\,(\ol{\nabla v})
                                + 2\,\ol{u}\,(\nabla u)\,v + 2\,u\,(\ol{\nabla u})\,v
                                + 2\,u\,(\nabla u)\,\ol{v} \big)
               \Big)\,,\nl
\nabla (B''(u)(v,w)) &= -2\,\ii\,\tfrac{1}{\eps}\th
              \big((\ol{\nabla u})\,v\,w + (\nabla u)\,\ol{v}\,w + (\nabla u)\,v\,\ol{w} + \ol{u}\,(\nabla v)\,w  \\
&\qquad\qquad~~ {} + u\,(\ol{\nabla v})\,w + u\,(\nabla v)\,\ol{w} + \ol{u}\,v\,(\nabla w)
                   + u\,\ol{v}\,(\nabla w) + u\,v\,(\ol{\nabla w})
              \big)\,.
\end{align*}
This implies
\begin{align*}
\Laplace B(u) &= -\ii\,\tfrac{1}{\eps}
               \Big((\Laplace U)\,u + 2\,(\nabla U) \cdot (\nabla u) + U (\Laplace u) \\
&\qquad\qquad {}  + \th\,\big(2\,\ol{u}\,(\nabla u) \cdot (\nabla u)
                            + 4\,u\,(\ol{\nabla u}) \cdot (\nabla u)
                            + 2\,|u|^2 (\Laplace u) + u^2 (\ol{\Laplace u}) \big)
               \Big)\,,
\end{align*}
\begin{align*}
\Laplace (B'(u)\,v) &= -\ii\,\tfrac{1}{\eps}
               \Big((\Laplace U)\,v + 2\,(\nabla U) \cdot (\nabla v) + U(\Laplace v) \\
&\qquad\quad {}      + \th\,\big(2\,|u|^2\,(\Laplace v) + u^2\,(\ol{\Laplace v})
                                + 4\,\ol{u}\,(\nabla u) \cdot (\nabla v)
                                + 4\,u\,(\ol{\nabla u}) \cdot (\nabla v) \\
&\qquad\qquad\quad {}                 + 4\,u\,(\nabla u) \cdot (\ol{\nabla v})
                                      + 2\,\ol{u}\,(\Laplace u)\,v + 2\,u\,(\ol{\Laplace u})\,v \\
&\qquad\qquad\quad {}                 + 2\,u\,(\Laplace u)\,\ol{v} + 2\,\ol{v}\,(\nabla u) \cdot (\nabla u)
                                      + 4\,v\,(\ol{\nabla u}) \cdot (\nabla u) \big)
               \Big)\,,
\end{align*}
\begin{align*}
\Laplace (B''(u)(v,w)) &=  -2\,\ii\,\tfrac{1}{\eps}\,\th
               \big((\ol{\Laplace u})\,v\,w + (\Laplace u)\,\ol{v}\,w + (\Laplace u)\,v\,\ol{w}
                     + \ol{u}\,(\Laplace v)\,w + u\,(\ol{\Laplace v})\,w + u\,(\Laplace w)\,\ol{w} \\
&\qquad\qquad~ {}    + \ol{u}\,v\,(\Laplace w) + u\,\ol{v}\,(\Laplace w) + u\,v\,(\ol{\Laplace w})  \\
&\qquad\qquad~ {}    + 2\,\ol{w}\,(\nabla u) \cdot (\nabla v)
                     + 2\,w\,(\ol{\nabla u}) \cdot (\nabla v) + 2\,w\,(\nabla u) \cdot (\ol{\nabla v}) \\
&\qquad\qquad~ {}    + 2\,\ol{v}\,(\nabla u) \cdot (\nabla w)
                     + 2\,v\,(\ol{\nabla u}) \cdot (\nabla w) + 2\,v\,(\nabla u) \cdot (\ol{\nabla w}) \\
&\qquad\qquad~ {}    + 2\,\ol{u}\,(\nabla v) \cdot (\nabla w) + 2\,u\,(\ol{\nabla v}) \cdot (\nabla w)
                     + 2\,u\,(\nabla v) \cdot (\ol{\nabla w})
               \big)\,.
\end{align*}
Higher derivatives of $B$, which appear in higher-order commutators,
can be expressed in a similar way but will not be listed here.

\subsection{Auxiliary results for the evolutionary operators $\EA$ and $\EB$} \label{Composition}
The evolutionary operators $\EA$ and $\EB$ are given by~\eqref{eq:E-def}.
For the nonlinear operator $\EB$, the Fr{\'e}chet derivatives with respect to the initial value $u$
are of increasing complexity:
\begin{align*}
\pdd\,\EB(t,u)\,v &= \ee^{-\ii\frac{t}{\eps} (U+\th\,|u|^2)} v
                     -\ii\,\tfrac{t}{\eps}\,\th\,\ee^{-\ii\frac{t}{\eps}(U+\th\,|u|^2)}\,u\,(\ol{u}\,v + u\,\ol{v})\,, \\
\pdd^2\,\EB(t,u)(v,w) &= -\ii\,\tfrac{t}{\eps}\,\th\,\ee^{-\ii\frac{t}{\eps}(U+\th\,|u|^2)}\\
&\qquad \cdot
                         \big(2\,u\,\ol{v}\,w + 2\,\ol{u}\,v\,w + 2\,u\,v\,\ol{w}
                               -\ii\,\tfrac{t}{\eps}\,\th\,u\,(\ol{u}\,v + u\,\ol{v})(\ol{u}\,w + u\,\ol{w}) \big)\,.
\end{align*}
For higher derivatives $\pdd^{k}$ the results look similar and involve higher powers
of $ \tfrac{t}{\eps} $. Furthermore,
\begin{align*}
\pdd\,\EB(t,u)\,\big|_{\,t=0} &= \mathrm{id}\,, \\
\pdd^2\,\EB(t,u)\,\big|_{\,t=0} &= 0\,.
\end{align*}
In the present situation, we may use the identity
\begin{equation*}
\nabla \EA(t)\,u = \EA(t) \nabla u\,,
\end{equation*}
thus
\begin{align*}
&\nabla \EB(t,\EA(t)\,u) = \ee^{-\ii\frac{t}{\eps}(U+\th\,|\EA(t)\,u|^2)} \\
&\qquad {} \cdot \Big(\EA(t)\nabla u
         -\ii\,\tfrac{t}{\eps}
           \big((\nabla U)\,\EA(t)\,u + \th\,|\EA(t)\,u|^2\,(\EA(t)\nabla u)
                 + \th\,(\EA(t)\,u)^2\,(\ol{\EA(t)\nabla u}) \big) \Big)\,,
\end{align*}
which implies
\begin{align*}
&\Laplace \EB(t,\EA(t)\,u) = \ee^{-\ii\frac{ t}{\eps}(U+\th\,|\EA(t)\,u|^2)} \Big(\EA(t)\,\Laplace u \\
& {} -\ii\,\tfrac{t}{\eps} \big(2\,(\nabla U) \cdot (\EA(t) \nabla u)
                                  + 2\th\,(\ol{\EA(t)\,u })(\EA(t)\nabla u) \cdot (\EA(t)\nabla u)
                                  + 3\th\,(\EA(t)\,u)\,|\EA(t) \nabla u |^2  \\
&\qquad {}                        + (\Laplace U)\,\EA(t)\,u + \th\,|\EA(t)\,u|^2\,\EA(t)\Laplace u
                                  + \th\,(\EA(t)\,u)^2\,\ol{\EA(t)\Laplace u } \big) \\
& {} -\tfrac{t^2}{\eps^2} \big((\nabla U) \cdot (\nabla U)\,\EA(t)\,u
                                  + 2\th\,(\nabla U) \cdot (\EA(t)\nabla u)\,|\EA(t)\,u|^2 \\
&\qquad {}                        + 2\th\,(\nabla U) \cdot (\ol{\EA(t)\nabla u})(\EA(t)\,u)^2
                                  + 2\th^2|\EA(t)\nabla u|^2\,|\EA(t)\,u|^2\,\EA(t)\,u  \\
&\qquad {}                        + \th^2 (\ol{\EA(t)\nabla u}) \cdot (\ol{\EA(t)\,\nabla u})(\EA(t)\,u)^3
                                  + \th^2 (\EA(t)\nabla u) \cdot (\EA(t) \nabla u)\,
                                    |\EA(t)\,u|^2\,\,\ol{\EA(t)\,u} \big) \Big)\,.
\end{align*}

\subsection{Representation of commutators} \label{commut}
The results derived above for the operators $A$ and $B$ 
yield explicit expressions for the relevant commutators from~\eqref{eq:ABcommutators}.
With $ \ii\,\eps\,(-\ii\,\frac{1}{\eps}) = 1 $ we obtain, for a general potential $ U $,
\begin{subequations}
\begin{align}
[A,B](u) &= (\nabla U) \cdot (\nabla u) + \half\,(\Laplace U)\,u
            + \th\,\big(u^2\,(\ol{\Laplace u}) + 2\,u\,|\nabla u|^2
            + \ol{u}\,(\nabla u) \cdot (\nabla u) \big)\,, \label{eq:AB-GPE} \nl
[A,B'(u)](v) &= \half\,(\Laplace U)\,v + (\nabla U) \cdot (\nabla v) \label{eq:AB'-GPE} \\
& \quad {}  + \th\,\big(u^2 (\Laplace \ol{v}) + (\Laplace u)\,\ol{u}\,v
                       + (\Laplace u)\,u\,\ol{v} + (\ol{\Laplace u})\,u\,v
                       + 2\,|\nabla u|^2\,v \notag \\
& \qquad\quad {}       + 2\,\ol{u}\,(\nabla u) \cdot (\nabla v) + 2\,u\,(\ol{\nabla u}) \cdot (\nabla v)
                       + (\nabla u) \cdot (\nabla u)\,\ol{v}
                       + 2\,u\,(\nabla u) \cdot (\ol{\nabla v}) \big)\,. \notag
\end{align}
Furthermore,
\begin{align}
&[B,[B,A]](u)  {} \label{eq:BBA-GPE} \\
& {} = -\ii\,\tfrac{1}{\eps}
        \Big(- u\,(\nabla U) \cdot (\nabla U) \notag \\
& \qquad\qquad {} + \th\,\big(2\,(\Laplace U)\,|u|^2\,u
                            - 4\,|u|^2\,(\nabla U) \cdot (\nabla u)
                            - 2\,U\,u^2\,(\ol{\Laplace u}) \big) \notag \\
& \qquad\qquad {} - \th^2\,\big(2\,|u|^4\,(\Laplace u)
                              - 2\,|u|^2\,u^2\,(\ol{\Laplace u})
                              + |u|^2\,\ol{u}\,(\nabla u) \cdot (\nabla u)
                              + 6\,|u|^2\,u\,|\nabla u|^2 \notag \\
&\qquad\qquad {} + u^3\,(\ol{\nabla u}) \cdot (\ol{ \nabla u}) \big)
        \Big)\,. \notag
\end{align}%
This expression comprises less critical terms with respect to $U$,
in particular it does not contain terms $ U(\nabla u) \cdot (\nabla u) $, or $ U(\Laplace u) $.

For $ [A,[A,B]](u) $, using the identities
\begin{align*}
\Laplace \big((\nabla u) \cdot (\nabla v)\big)
&= \big(\nabla\,(\Laplace u) \big) \cdot (\nabla v)
   + (\nabla u) \cdot \big(\nabla\,(\Laplace v) \big)
   + 2 \Tr \big((\nabla \nabla^T u) \cdot (\nabla \nabla^T v) \big)\,, \\
(\nabla u) \cdot \big(\nabla\,((\nabla u) \cdot (\nabla u)) \big)
&= 2\,(\nabla u)^T \cdot (\nabla \nabla^T u) \cdot (\nabla u)\,,
\end{align*}
we obtain\footnote{For the harmonic potential $ U $ from~(\ref{eq:harmPot}),
                   the terms $ \Laplace^2 U $ and $ \nabla (\Laplace U) $ vanish.}
\begin{align}
&[A,[A,B]](u)  {} \label{eq:ABA-GPE} \\
& {} = \ii\,\eps \Big(\tfrac{1}{4}(\Laplace^2 U) u +  (\nabla (\Laplace U)) \cdot (\nabla u)
+  \Tr \big((\nabla \nabla^T (U)) \cdot (\nabla \nabla^T (u)) \big) \notag \\
&\qquad\quad {} + \th\,\big(u^2\,(\ol{\Laplace^2 u})
                        + 4\,u\,(\nabla u) \cdot (\nabla (\ol{\Laplace u}))
                        + 2\,(\ol{\Laplace u}) (\nabla u) \cdot (\nabla u) \big) \Big) \notag \\
&\quad {} + \ii\,\eps\,\th\,\Big(\ol{u}\,\Tr\big((\nabla \nabla^T (u))^2\big)
                                  + 2\,u \Tr\big((\nabla \nabla^T (u)) \cdot (\ol{\nabla \nabla^T (u)})\big)
                                  + 2\,(\ol{\nabla u})^T \cdot (\nabla \nabla^T (u)) \cdot (\nabla u) \notag \\
&\qquad\qquad~ {} + 2\,(\nabla u)^T \cdot (\ol{\nabla \nabla^T (u)}) \cdot (\nabla u)
                  + 2\,(\nabla u)^T \cdot (\nabla \nabla^T (u)) \cdot (\ol{\nabla u}) \Big)\,. \notag
\end{align}
\end{subequations}

\section{$L^2$\,-\,estimate of the local error for the cubic Schr{\"o}dinger equation} \label{Lestim}
Since solutions to Schr{\"o}dinger equations are well-defined in the Hilbert space $L^2$,
we aim for an $L^2$\,-\,estimate of $\nL(t,u)$ on the basis of the general representation
from Section~\ref{sec:locerr-repr}. Proceeding from~\eqref{eq:L-double-integral},
the local error terms $ \L{2}(t,u) $ and $ \L{1,1}(t,u) $ will be estimated separately \bcor{in the situation of~\eqref{eq:Problem} below}.
The detailed derivations of these estimates are given in Appendix~\ref{lab:gronwall}.

\subsection{$ L^2 $\,-\,estimates for \,$ \L{2} $ and $ \L{1,1} $}
Consider
\begin{subequations}
\begin{equation} \label{eq:L2integral}
\L{2}(t,u) = \int_{0}^{t} \int_{0}^{\t_1} \pdd\,\EF(t-\t_2,\nS(\t_2,u))\,\S2(\t_2,u)\, \dd \t_2\, \dd \t_1
\end{equation}
with $ \S2(\t_2,u) = \int_{0}^{\t_2} \s2(\t_2,\t_3,u)\,\dd\t_3$ given by~\eqref{eq:S2integral}.
In combination with an estimate for $\s2$, the integrand in~\eqref{eq:L2integral} can be estimated by
\begin{align}
\label{eq:pdEF}
& \big\| \pdd\,\EF(t -\t_2,\nS(\t_2,u)) \cdot\S2(\t_2,u) \big\|_{L^2} \nl
& \quad {} \leq \exp \Big(C \int_{\t_2}^{t} \tfrac{1}{\eps}
                          |\th|\,\,\|\EF(\sig -\t_2,\nS(\t_2,u))\|_{H^2}^2\,\dd \sig \Big)
                          \Big(t \,\cdot \sup_{0 \leq \t_3 \leq \t_2} \| \s2(\t_2,\t_3,u)\|_{L^2}
                                + \tfrac{t}{\eps}\,C_\ast \Big)\,,
\notag
\end{align}
\end{subequations}
with a constant $C_\ast$ as indicated in Appendix~\ref{C1}.

The second contribution to the local error~\eqref{eq:L-double-integral} is
\begin{subequations}
\begin{equation} \label{eq:L11integral}
\L{1,1}(t,u) = \int_{0}^{t} \int_{0}^{\t_1}
 \pdd^2\,\EF(t-\t_2,\nS(\t_2,u))\,\big(\S1(\t_2,u),\S1(\t_2,u) \big)\Big\}\,\dd\t_2\,\dd\t_1
\end{equation}
with $ \S1(\t_2,u) = \int_{0}^{\t_2} \s1(\t_2,\t_3,u)\,\dd\t_3$ given by~\eqref{eq:S1integral}.
The calculations from Appendix~\ref{lab:gronwall-1} yield
\begin{align}
\label{eq:pd2EF}
\big\| \pdd^2 \nE_F(t-\t_2,\nS(\t_2,u))& (\S1(\t_2,u),\S1(\t_2,u)) \big\|_{L^2}     \\
& {} \leq \exp\Big( C \int\limits_{\t_2}^{t}
                                      \tfrac{1}{\eps}\,|\th|\,\|\nE_F(\sig-\t_2,\nS(\t_2,u)) \|_{H^2}^2\,\dd\sig \Big)\notag \\
&    \qquad {} \cdot \Big(\hat{C}\,\tfrac{t^3}{\eps}\,\|u \|_{L^2}\,
                      \big(\sup_{0 \leq \t_3 \leq \t_2}\|\s1(\t_2,\t_3,u) \|_{H^2}
                      + \tfrac{1}{\eps}\,C_\ast \big)^2
                      + \tfrac{t^2}{\eps}\,C_\ast\Big)
                      \,\,.
\notag
\end{align}
\end{subequations}
In this way, estimation of the local error reduces to estimates for $ \s1 $ and $ \s2 $ which
will be discussed in Sections~\ref{subsec:s1est} and~\ref{subsec:s2est}.

\subsection{$ L^2 $\,-\,boundedness of $U\EA(t)\,u$ for the harmonic potential $U$ from~\eqref{eq:harmPot}}
The product of the quadratic potential $ U(x)=\half\,\omega^2\,|x|^2 $ with the function $\EA(t)\,u$
as it appears in~\eqref{eq:S1raw} below and in the estimates from Section~\ref{Hm-stability},
is unbounded in general. In the following we work out requirements on $u$
which guarantee that $ \| U \EA(t)\,u \|_{L^2} $ remains bounded.

From the well-known identity
\begin{align*}
\EA(t)\,x \,\EA(-t) &= x+ \ii\, t\, \eps\, \nabla
\end{align*}
we first obtain
\begin{align*}
\| x\,\EA(t) u \|_{L^2} & \leq \| x u\|_{L^2} + \eps\, t\, \| u \|_{H^1}\,.
\end{align*}%
Furthermore, using the estimates from Appendix~\ref{app:A3} we find 
\begin{align}
\| U \EA(t) u \|_{L^2}
& \leq \| U\,u \|_{L^2} + t\,\eps\,C_1 \,\|(\nabla U)\,u \|_{H^1} + t^2 \, \eps^2 \, C_2\, \| u \|_{H^2}\notag \\
& \leq \| U\,u \|_{L^2} + t\,\eps\,C_1 \,\big(\|U\,u \|_{L^2} + \| u \|_{H^2}\big)+ t^2 \, \eps^2 \, C_2\, \| u \|_{H^2},
\end{align}
with $C_1,\,C_2$ depending on the weight $\omega$ in $U$.
Clearly, the expression $\|(\Laplace U)\,u \|_{L^2}$ is bounded by $C\,\| u \|_{L^2}$
with a constant $C$ depending on $\omega$.

\subsection{$ H^m $\,-\,boundedness of a Strang splitting step} \label{Hm-stability}
In the estimates of $\S1$ and $\S2$, certain $H^m$\,-\,norms of the splitting
approximation ${\nS(t,u)=\EA(\half\,t) \cdot \EB(t,\EA(\half\,t)\,u)}$
and of the intermediate composition $\EB(t,\EA(\half\,t)\,u)$ occur.
Their boundedness with respect to the initial value $ u $ is critical for our analysis.
Due to the invariance property $\| \EA(t)\,u \|_{H^m} = \| u \|_{H^m}$,
the expressions $ \nS(t,u) $ and $w= \EB(t,\EA(\half\,t)\,u) $
show the same behavior in the $H^m$-norm.

Concerning $m=0$, both flows $\EA$ and $\EB$ conserve the $L^2$\,-\,norm, hence
\begin{equation}
\| \nS(t,u) \|_{L^2} = \| \EB(t,\EA(\half\,t)\,u)\|_{L^2} = \| u \|_{L^2}\,.
\end{equation}
The following estimates are based on the results from~Section~\ref{Composition}, making use of the estimates from Appendix~\ref{app:A2}.
For $m=1$ and $ m=2 $ we have
\begin{subequations}
\begin{equation}
\| \EB(t,\EA(\half \, t)\, u) \|_{H^1}
\leq \| u \|_{H^1} + C \cdot \Big( \tfrac{t}{\eps}\,
									\big( \|(\nabla U)\,u \|_{L^2} + |\th|\,\| u \|_{H^2} \| u \|_{H^1}^2\big) +
									t^2  \| u \|_{H^1} \Big)\,,
\end{equation}
and
\begin{align}
\| \EB(t,\EA(\half\,t)\,u)\|_{H^2}
&\leq \| u \|_{H^2}
      + C \cdot \Big(\tfrac{t}{\eps}\,\big(\| U\,u \|_{L^2}+ \| u \|_{H^2}
      + |\th|\,\| u \|_{H^2}^3 \big) + t^2 \| u \|_{H^2}  \\
& \qquad {}           + \tfrac{t^2}{\eps^2}\,
                        \big( \| U  u \|_{L^2} + | \th| \|U\,u \|_{L^2} \| u \|_{H^2}^2
                        	 + |\th|\,\|u\|_{H^2}^3 + |\th|^2 \|u\|_{H^2}^5\big) \notag \\
& \qquad  {} 		  +  \tfrac{t^3}{\varepsilon} \,
						\big( \| U\,u\|_{L^2} + \| u \|_{H^2} + |\th|\,\| u \|_{H^2}^3 \big) +
						t ^4 \| u \|_{H^2}
                \Big)\,. \notag
\end{align}%
\end{subequations}
Analogous estimates for higher Sobolev indices $m$
involve powers up to $\big(\tfrac{t}{\eps} \big)^m$ and higher Sobolev norms of
$u$ as well as $U\,u$.

\subsection{$ H^2 $\,-\,estimate for $\s1(t,\t,u)$} \label{subsec:s1est}
The integrand $ \s1 $ in the integral representation~\eqref{eq:S1integral} for $ \S1 $
can be estimated by
\begin{align*}
\|\s1(t,\t,u)\|_{H^2} & \leq C\,\big\| [A,B](\EA(\half \t)\,\EB(t,\EA(\half\,t)\,u)) \big\|_{H^2}\\
&\quad {} + C\,\big\|\pdd\,\EB(t-\t,\EB(\t,\EA(\half\,t)\,u))\,[B,A](\EB(\t,\EA(\half\,t)\,u)) \big\|_{H^2}\,, \notag
\end{align*}
and further
\begin{align}
\label{eq:S1raw}
\big\| \s1(t,\t,u) \big\|_{H^2}&
\leq C\,\big(1 + \tfrac{t}{\eps}\,|\th|\,\| w \|_{H^2}^2 \big) \cdot \Big(\big( 1 + \tfrac{t}{\eps}|\th|\,\| u \|_{H^2}^2\big)^2\,\| [A,B](w)\|_{H^2} \\
&\qquad + t\, \eps \| w\|_{H^4}  + \tfrac{t}{\eps} \big( \| U \Delta w \|_{L^2} + \| (\nabla U )\cdot (\nabla w)\|_{L^2}  \notag\\
& \qquad \qquad \qquad  \qquad \quad +|\th|\,\| (\nabla U )\cdot (\nabla w) \|_{L^2}\| w \|_{H^4} \| w \|_{H^3}\big) \notag\\
&  \qquad + \tfrac{t^2}{\eps^2} \big( \| U (\nabla U )\cdot (\nabla w) \|_{L^2} + \| U w\|_{L^2} + |\th|\,\| U w\|_{L^2} \| w \|_{H^4} \| w \|_{H^3} \notag\\
& \qquad \qquad \quad  + |\th|\,\| U\,u \|_{L^2} \| u \|_{H^2} \| [A,B](w)\|_{H^2}\notag \\
& \qquad \qquad \quad + t\, \eps \| u \|_{H^2}^2 \| [A,B](w) \|_{H^2}\big)
\Big)\;
                           \Big|_{\,\begin{subarray}~~v = \EA(\shalf\,\t)\,u  \\
                                   ~w = \EB(\t,\EA(\shalf\,\t)\,u) \end{subarray}}\,,\notag
\end{align}
where $(\nabla U)\cdot(\nabla U) = C\,U$ and
\begin{align*}
\| [A,B](w)\|_{H^2}& \leq C\, \big( \| (\nabla U) \cdot (\nabla w)\|_{H^2}  + |\th|\,\| w\|_{H^4} \| w\|_{H^3}\| w\|_{H^2}  \big)\,.
\end{align*}
Inserting the expressions for $v$ and $w$ in~\eqref{eq:S1raw} we obtain
\begin{align}
&\sup_{0 \leq \t \leq t} \|\s1(t,\t,u)\|_{H^2} \label{eq:s1est} \\
& \quad{} \leq C_1\, \Big( \| U^2\,u \|_{L^2} +\| U\, u\|_{L^2}+\| u \|_{H^4} +   |\th|\,\| u \|_{H^4} \| u \|_{H^3} \| u \|_{H^2} \Big) \notag \\
& \qquad + C_2\, \tfrac{t}{\eps} \Big( \| U\,u\|_{L^2} + \| u \|_{H^2}+  |\th|\,\| U^2 u \|_{L^2} \| u \|_{H^2}^2 + |\th|\,\| U\,u \|_{L^2} \| u \|_{H^4} \| u \|_{H^3} \notag \\
& \qquad \qquad \quad  + |\th|\,\| u \|_{H^4} \| u \|_{H^2}^2  + |\th|^2 \| u \|_{H^4} \| u \|_{H^3} \| u \|_{H^2}^3
	+ \eps^2 \big( \| u \|_{H^4} + \| U^2 u \|_{L^2} \big) \Big) \notag \\
& \qquad + C_3\, \tfrac{t^2}{\eps^2} \Big( \| U^2 u \|_{L^2} + \| U\,u \|_{L^2}+
		\| u \|_{H^4}+ |\th|\,\| U^2\,u\|_{L^2} \| u \|_{H^2}^2 \notag \\
& \qquad \qquad \quad +|\th|\,\| U\, u \|_{L^2}^2 \| u \|_{H^2}+ |\th|\,\| U\,u\|_{L^2} \| u \|_{H^4} \| u \|_{H^3} \notag \\
& \qquad \qquad \quad   +|\th|\,\| u \|_{H^4}\| u \|_{H^2}^2+ |\th|^2 \| U\,u\|_{L^2} \| u \|_{H^3}^2 \| u \|_{H^2}^2 \notag \\
& \qquad \qquad \quad + |\th|^2 \|u \|_{H^4}\|u\|_{H^3}\|u\|_{H^2}^3   + |\th|^3 \| u\|_{H^4} \| u \|_{H^3} \| u\|_{H^2}^5 \notag \\
& \qquad \qquad \quad+ \eps^2 \big( |\th|\,\| U\,u\|_{L^2} \| u \|_{H^4} \| u \|_{H^2}  + |\th|\,\| u \|_{H^4} \| u \|_{H^3}^2\big) + \eps^4 \| u \|_{H^4} \Big) \notag \\
& \qquad + \bcor{\nO\big(\tfrac{t^3}{\eps^3}\big)\big(1+\Order(\eps^2)\big)}\,. \notag
\end{align}

\subsection{$ L^2 $\,-\,estimate for $\s2(t,\t,u) $} \label{subsec:s2est}
The integrand $ \s2 $ in the integral representation~\eqref{eq:S2integral} for $ \S2 $
involves multiple commutators and derivatives of the flows $\EA$ and $\EB$.
Here we only note the dominant terms according to~\eqref{eq:s2dominant} in more detail:
\begin{equation*}
\| \s2(t,\t,u) \|_{L^2} \leq \half\,\| [B,[B,A]](u) \|_{L^2} + \tfrac{1}{4}\,\| [A,[B,A]](u) \|_{L^2} + \bcor{\Order(\tfrac{t}{\eps^2}+ t)}\,,
\end{equation*}
where the dominant commutators, given by~\eqref{eq:BBA-GPE} and~\eqref{eq:ABA-GPE},
can be estimated by
\begin{align*}
\| [B,[B,A]](u) \|_{L^2}
& \leq C\, \tfrac{1}{\eps} \Big( \| U\,u \|_{L^2} +|\th|\,\| U\,u \|_{L^2}\, \| u\|_{H^4}\,\| u \|_{H^2} + |\th|\,\| u \|_{H^2}^3 + |\th|^2\, \|u\|_{H^2}^5\Big)\,,\\
\| [A,[B,A]](u) \|_{L^2}
& \leq C\,\eps\, \Big( \| u \|_{H^2} + |\th|\, \| u \|_{H^4}\,  \bcor{\| u \|_{H^2}^2} \Big)\,.
\end{align*}
A more refined estimate reads 
\begin{align}
&\sup_{0 \leq \t \leq t} \|\s2(t,\t,u)\|_{L^2} \label{eq:s2est}  \\
& \quad {} \leq  C_1\, \eps \Big(\| u \|_{H^2} +|\th|\, \| u \|_{H^4}\,\| u \|_{H^2}^2\Big) \notag \\
& \qquad {} + C_2\, \tfrac{1}{\eps} \Big( \| U\, u\|_{L^2} +
			|\th|\,\|U\, u\|_{L^2} \, \|u \|_{H^4} \, \| u \|_{H^2} +
			|\th|\,\| u \|_{H^2}^3 + |\th|^2\, \| u\|_{H^2}^5\Big) \notag \\
& \qquad {} + C_3\, t\, \Big(\| U^2\, u\|_{L^2} + \| U\, u \|_{L^2}+ \| u \|_{H^4} +
		 |\th|\,\| U\,u \|_{L^2}\, \| u\|_{H^4}\,  \|u\|_{L^2} \notag \\
& \qquad \qquad \quad  + |\th|\, \|u\|_{H^4}\, \| u \|_{H^2}^2 +
		 |\th|\, \| u\|_{H^3}^2 \, \| u \|_{H^2} +
		 |\th|^2\, \| u \|_{H^4}\, \| u \|_{H^2}^4\Big) \notag \\
&\qquad {} + C_4\,t\, \tfrac{1}{\eps^2}\, \Big( |\th|\, \| U \, u\|_{L^2}^2 \, \| u \|_{H^4} +
		|\th|\, \| U \, u\|_{L^2} \,\|u \|_{H^4} \, \| u\|_{H^2} \notag \\
& \qquad \qquad \qquad + |\th|^2\, \| U\, u\|_{L^2} \, \|u \|_{H^2}^4+ |\th|^3\, \| u\|_{H^2}^7\Big) \notag \\
& \qquad + \bcor{\Order\big( \tfrac{t^2}{\eps^3} + \tfrac{t^2}{\eps} + t^2\,\eps\big)}\,. \notag
\end{align}

\subsection{Resulting $ L^2 $\,-\,estimate for $ \nL(t,u) $}
Combining all previous estimates we conclude
\begin{align}
\label{eq:LocErr}
\big\| \nL(t,u) \big\|_{L^2}
&\leq \int_{0}^{t} \int_{0}^{\t_1} \Big\{ \int_{0}^{\t_2}
       \|  \pdd\,\EF(t-\t_2,\nS(\t_2,u))\,\s2(\t_2,\t_3,u) \|_{L^2}\,\dd\t_3 \notag \\
& \qquad\qquad {} +\Big\| \pdd^2\,\EF(t-\t_2,\nS(\t_2,u)) \cdot \Big(\int_{0}^{\t_2} \s1(\t_2,\t_3,u)\,\dd \t_3 \Big)^2 \Big\|_{L^2} \Big\}\,\dd\t_2\,\dd\t_1 \notag \\
& \leq\ \tilde{C}  \cdot t^3 \exp \Big(C\,\tfrac{t}{\eps}\,|\th|\,\sup_{0 \leq \chi \leq \sig \leq t} \|\EF(\sig,\nS(\chi,u)) \|_{H^2}^2   \Big)  {} \\
&\cdot  \sup_{0 \leq \t \leq t} \Big(\tfrac{t^2}{\eps}\,\| u \|_{L^2} \big(\| \s1(t,\t,u) \|_{H^2}
                                  + \tfrac{1}{\eps}\,C_\ast \big)^2
                                + \| \s2(t,\t,u) \|_{L^2}
                                +  \tfrac{1}{\eps}\,C_\ast + \tfrac{t}{\eps}\,C_\ast\Big)\,, \notag
\end{align}
where $\s1$ and $\s2$ have been estimated in Sections~\ref{subsec:s1est} and~\ref{subsec:s2est}.

\subsection{$ L^2 $\,-\,estimate for $ \nL(t,u) $ for the Lie splitting method}
For comparison, we recapitulate an $L^2$\,-\,estimate for the Lie splitting method
\begin{equation} \label{eq:Lie}
\nS(t,u) = \nS_{\text{Lie}}(t,u) = \EB(t,\EA(t)\,u)\,.
\end{equation}
Following~\cite{Auz14} we have, for $ A $ linear,
\begin{align}
\label{eq:errLiesplitting}
\begin{aligned}
\nL(t,u)  = \int_0^t \int_0^{\t_1}
\pdd \EF(t-\t_1,\nS(\t_1,u)) \pdd & \EB (\t_1-\t_2,\EB(\t_2,\EA(\t_1) u))  \\
& \qquad {}\cdot [B, A] (\EB (\t_2 ,\EA (\t_1)u)) \,\dd\t_2 \,\dd \t_1 \,.
\end{aligned}
\end{align}
Evaluating $ \nL(t,u) $ at $ t=0 $ reveals the dominant term in its Taylor expansion,
\begin{equation}
\label{eq:Lie err-dominant}
\nL(t,u) = \tfrac{t^2}{2}\,[B,A](u) +\bcor{ \Order\big(\tfrac{t^3}{\eps}+ t^3\, \eps\big)}\,.
\end{equation}
Proceeding similarly as for the case of the Strang splitting method we obtain
\begin{align}
\label{eq:LocErrLie}
\| \nL (t,u) \|_{L^2}  \leq&\,
\tilde {C} \cdot t^2 \exp \Big(C  \tfrac{t}{\eps}  |\th|\,\sup_{0 \leq \chi \leq \sig \leq t} \|\EF(\sig,\nS(\chi,u)) \|_{H^2}^2   \Big)  {} \\
& {} \cdot\Big( \| U \, u\|_{L^2} + \| u \|_{H^2} + |\th|\,\|u\|_{H^2}^3+ t\, \eps \, \| u\|_{H^2} \notag \\
& \quad {}+ \tfrac{t}{\eps}\big( \| U\, u\|_{L^2} + |\th|\,\| U \, u\|_{L^2} \| u\|_{H^2} \| u \|_{H^2} + |\th|\,\| u \|_{H^2} \| u \|_{H^1}^2 \big) \notag\\
& \quad {}+ t^2 \big( \| U\, u\|_{L^2} + \| u\|_{H^2} + |\th|\,\| u\|_{H^2}^3 \big)  \notag\\
& \quad {}+\tfrac{t^2}{\eps^2}\, |\th| \big( \| U\, u\|_{L^2} \| u \|_{H^2}^2 + \|U\, u\|_{L^2}^2 \| u\|_{H^2} + |\th|\,\| U \, u\|_{L^2} \| u\|_{H^2}^4\notag \\
&\quad\qquad \qquad  {} + |\th|\,\|u\|_{H^2}^5 + |\th|^2 \| u\|_{H^2}^7\big) +\tfrac{1}{\eps} C_\ast \Big)  \notag \\
&  +\bcor{\Order\big(\tfrac{t^5}{\eps^3} + \tfrac{t^5}{\eps}+t^5\, \eps\big)} \,.	\notag
\end{align}%
Here the constant $C_\ast$ discussed in Appendix~\ref{C1} again appears.
We observe that the dominant $ \Order(t^2) $ term of the local error
does not depend on $\eps$, which is a different behavior compared to the Strang splitting method.

\subsection{Higher-order methods}
For the Lie and Strang splitting methods, the leading term of the local error
has a very simple structure and is influenced by $[B,A]$ and $[A,[B,A]],\,[B,[B,A]]$, respectively
(see~\eqref{eq:Lie err-dominant},\,\eqref{eq:Strang err-dominant}).
It is also known that for a higher-order method the leading term in the Taylor
expansion of the local error comprises iterated commutators,
see for instance~\cite{auzinger2014local} for a precise discussion.

However, as we have seen, an {\em exact}\, (integral) representation of the local error becomes quite
complicated already for the Strang case, and it involves various derivatives of the nonlinear
operator $ B(u) $ and derivatives of commutator expressions.
Due to this significant increase in complexity,
a rigorous analysis of higher-order splitting methods~\eqref{eq:splitting-s-stages}
appears to be a major challenge
for the nonlinear case which we do not attempt to cope with
here.\footnote{See~\cite{auzingeretal13a} for an exact local error representation for the
                 linear case. This is combinatorially rather involved;
                 however, the role of iterated commutators dominating the error is clearly worked~out.}

Nevertheless we may infer information about the behavior of the dominant terms.
For a third-order scheme, for instance,
the local error will be dominated by the commutator expressions
\begin{equation*}
[A,[A,[A,B]]](u)\,,\quad [A,[B,[A,B]]](u) = [B,[A,[A,B]]](u)\,,\quad  [B,[B,[A,B]]](u)\,.
\end{equation*}
For the NLS we have $[B,[B,[A,B]]]=0$ and therefore no terms involving $\frac{1}{\eps^2}$ appear.
Moreover, all other third-order commutators have either no or a quadratic dependence on $\eps$.
This is comparable to our results for the Lie splitting method.
This means that especially for $\eps \ll 1$,
one would see the classical behavior or even a better
order behavior for more regular initial values.
Numerical observations are reported in Section~\ref{sec:numerics}.

\bcor{\paragraph{Summarizing remark.}
The theoretical analysis has shown a delicate dependence of the numerical error on the semiclassical parameter and the stepsize. Consequently for practical computations, automatic stepsize control seems more promising than an attempt to choose the time steps a priori. To this end, a posteriori estimators for the local time stepping error are required. In the following, we construct and analyze an asymptotically correct local error estimator based on the defect of the splitting solution.}

\section{An a~posteriori error estimator} \label{sec:apost}
In~\cite{Auz14} the following a~posteriori local error estimator for
$ s $\,-\,fold splitting methods~\eqref{eq:splitting-s-stages} of order $ p $ was proposed;
it is based on an Hermite quadrature approximation for the local error integral~\eqref{eq:local-error-integral}.
In Proposition~\ref{prop:defect-based-estimate} we state how this error estimator
can be practically computed via an appropriate evaluation of the defect.

\begin{proposition}[\cite{Auz14}] \label{prop:defect-based-estimate}
Let $v_i$, $w_i$ be defined as
\begin{equation}
v_i = \EA(a_i\,t, w_{i-1})\,,\quad  w_i = \EB(b_i\,t, v_i)\,,\quad 1 \leq i \leq s\,,
\end{equation}
with $w_0=u$ and $w_s = \nS (t,u)$,
and consider the local a~posteriori error estimator for a method of order $p$ defined by
\begin{equation}
\label{eq:ErrEst}
\nP(t,u) = \tfrac{1}{p+1}t \S{1} (t,u) \approx \nL(t,u),
\end{equation}
involving the defect
\begin{equation}
\label{eq:defect-apost}
\S{1}(t,u) = \pdt{} \nS (t,u) - F(\nS(t,u))\,.
\end{equation}
This can be computationally evaluated in the following way:
\begin{equation}
\S{1}(t,u) = g^{(s)} \circ g^{(s-1)} \circ \ldots \circ g^{(1)} \circ g^{(0)} - F(w_s)\,,
\end{equation}
where
\begin{align*}
g^{(i)} (z) &= b_i \, B(w_i) + \pdd\,\EB(b_i\, t, v_i) \pdd\,\EA(a_i\, t, w_{i-1})
\big[ a_i\, A(w_{i-1}) + z \big]\,,\quad i \geq 1\,, \\
g^{(0)} (z) &= 0\,.
\end{align*}
Hence, $ \nP(t,u) $ can be computed as
\begin{equation}
\label{eq:aposteriori}
\nP = \tfrac{1}{p+1}\,t \Big(g^{(s)} \circ  \ldots  \circ g^{(0)} - F(w_s) \Big)\,.
\end{equation}
\end{proposition}
Since the $g^{(i)}$ depend on $v_i$, $w_i$ and $w_{i-1}$ only,
they can be evaluated in parallel with the splitting scheme
without the need to store all intermediate values $v_i$, $w_i$.

For problem~\eqref{eq:Problemcomplete} with $ \eps=1 $, the asymptotical order
\begin{equation*}
\nP(t,u) - \nL(t,u) = \Order(t^{p+2})
\end{equation*}
has been proven in~\cite{Auz14} for the cases $p=1$ (Lie) and $p=2$ (Strang).
We will extend this study by incorporating the dependence on $\eps<1$,
while for higher-order methods we resort to numerical observations.
Understanding the asymptotical order behavior is essential
to ensure reliability of a time-adaptive method.

\subsection{Analysis of the deviation of the a~posteriori error estimator}

\paragraph{Lie Splitting.}
For the deviation $\nP - \nL$ of the Lie splitting error estimator an integral
representation has been derived in~\cite{Auz14},
\begin{equation}
\label{eq:Est-err}
\nP(t,u) - \nL(t,u) = \int_{0}^{t}
                      \big(K_1(\t,t)\,\nG_1(\t,t,u) - K_2(\t,t)\,\nG_2(\t,t,u) \big)\,\dd \t\,,
\end{equation}
where $K_1$ and $K_2$ are the first- and second-order Peano kernels associated with the error of
the underlying trapezoidal quadrature,
\begin{equation*}
K_1(\t,t) = \t-\half\,t\,,\quad  K_2(\t,t) = \half\,\t\,(t-\t)\,,
\end{equation*}
and
\begin{align*}
\nG_1(\t,t,u) &=
\Big\{ \pdd\,\nE_F(t-\t,\nS(\t,u)) \int_{0}^{\t} \pdd\,\EB(\t -\t_2,\EB(\t_2,v)) \\
& \qquad \cdot\Big(B''(\EB(\t_2,v)) \big(\widetilde{\nS}^{(1)}(\t_2,v) \big)^2 + [[B,A],A](\EB(\t_2, v))
              + [[B,A],B](\EB(\t_2,v))  \\
& \qquad\qquad {} + 2\,[B,A]'(\EB(\t_2,v))\,\widetilde{\nS}^{(1)}(\t_2,v) \Big)\,\dd \t_2\\
& \qquad {} + \pdd^2 \nE_F(t-\t,\nS(\t,u)) \big(\S1(\t,u) \big)^2 \Big\}~\Big|_{~v=\EA(\t)u}\,, \nl
\nG_2(\t,t,u)
&= \Big\{ \pdd\,\nE_F(t-\t,\nS(\t,u))
          \Big(\pdd^2\,\EB(\t,v)(A\,v,[B,A](v)) \\
&\qquad \qquad \qquad \qquad \qquad \quad {} + \big(\pdd\,\EB(\t,v)\,A - A\,\pdd\,\EB(\t,v) \big)\,[B,A](v) \Big) \\
& \qquad {} + \pdd^2 \nE_F(t-\t,\nS(\t,u))\,\big(\S1(\t,u),\pdd\,\EB(\t,v)\,[B,A](v) \big) \\
& \qquad {} + \pdd\,\nE_F(t-\t,\nS(\t,u))\,\pdd\,\EB(\t,v)\,[[B,A],A](v) \Big\}~\Big|_{~v=\EA(\t)u}\,.
\end{align*}
Here, $\S1$ is the defect~\eqref{eq:defect-apost} which satisfies the integral representation
\begin{align*}
\S1(t,u) &= \widetilde{\nS}^{(1)}(t,\EA(t,u))\,,\\
\text{with} \quad \widetilde{\nS}^{(1)}(t,v) &= \int_{0}^{t} \pdd\,\EB(t-\t,\EB(\t,v))\,[B,A](\EB(\t,v))\,\dd \t\,.
\end{align*}
Denoting
	\begin{align*}
	T_1 &= \int_{0}^{\t} \pdd\,\EB(\t -\t_2,\EB(\t_2,v))
	\Big(B''(\EB(\t_2,v)) \big(\widetilde{\nS}^{(1)}(\t_2,v) \big)^2  + [[B,A],A](\EB(\t_2, v)) \\
	& \qquad + [[B,A],B](\EB(\t_2,v)) + 2\,[B,A]'(\EB(\t_2,v))\,\widetilde{\nS}^{(1)}(\t_2,v) \Big)\,\dd \t_2\,, \\
	T_2 &= \pdd^2\,\EB(\t,v)(A\,v,[B,A](v))
	+ \big(\pdd\,\EB(\t,v)\,A - A\,\pdd\,\EB(\t,v) \big)\,[B,A](v) \\
	& \qquad {} + \pdd\,\EB(\t,v)\,[[B,A],A](v)\,,\\
	T_{22} &= \pdd\,\EB(\t,v)\,[B,A](v) \,,
	\end{align*}
	we have 

\begin{align*}
\| \nP&(t,u) - \nL(t,u) \|_{L^2} \\
& {}\leq C\, t^2 \big( \| \pdd\,\nE_F(t-\t,\nS(\t,u))\,T_1 \|_{L^2}+  \|\pdd^2 \nE_F(t-\t,\nS(\t,u)) (\S1(\t,u))^2\|_{L^2}\big)\\
 &\quad{} + C\, t^3 \big( \| \pdd\,\nE_F(t-\t,\nS(\t,u))\,T_2 \|_{L^2} + \|\pdd^2 \nE_F(t-\t,\nS(\t,u)) (\S1(\t,u), T_{22} )\|_{L^2}\big)\,.
\end{align*}
Estimating $\pdd\,\nE_F$ and $\pdd^2 \nE_F$ as in Appendix~\ref{lab:gronwall}, we obtain
\begin{align*}
&\| \nP(t,u) - \nL(t,u) \|_{L^2}\leq C_1\,
    e^{\big( C_2\,\tfrac{t}{\eps}\,|\th|\,\sup_{0 \leq \chi \leq \sig \leq t}
    \|\EF(\sig,\nS(\chi,u)) \|_{H^2}^2 \big)} \\
&\quad\cdot\sup_{0 \leq\t_2\leq \t \leq t} \Big(t^2 \| T_1\|_{L^2}+ \tfrac{t^3}{\eps}C_\ast + \tfrac{t^5}{\eps} \| u\|_{L^2} \big( \|\s1(\t,\t_2,u)\|_{L^2} + \tfrac{1}{\eps}C_\ast\big)^2 + \tfrac{t^4}{\eps}C_\ast + t^3 \| T_2 \|_{L^2}\\
& \qquad {} + \tfrac{t^4}{\eps}C_\ast + \tfrac{t^5}{\eps} \| u\|_{L^2} \big( \|\s1(\t,\t_2,u)\|_{L^2} + \tfrac{1}{\eps}C_\ast\big)(\| T_{22}\|_{L^2} + \tfrac{1}{\eps}C_\ast) + \tfrac{t^4}{\eps}C_\ast\Big)\,.
\end{align*}
Now we separately estimate,
\begin{align*}
\| T_1\|_{L^2} & \leq t\,\eps \big( \| u \|_{H^2} + |\th|\,\| u \|_{H^4} \| u \|_{H^3} \| u \|_{H^2} \big)\\
& \quad {} + \tfrac{t}{\eps}\big( \| U \, u \|_{L^2} + |\th|\,\| U u\|_{L^2} \| u \|_{H^4} \| u \|_{H^2} + |\th|\,\|u\|_{H^2}^3 + |\th|^2 \| u \|_{H^2}^5\big) \\
& \quad {} + t^2\big( \| U^2 u\|_{L^2}+ \| U \, u \|_{L^2} + \| u \|_{H^4} + |\th|\,\| u \|_{H^3}\| u \|_{H^2}^2 + |\th|\,\| U \, u \|_{L^2} \| u \|_{H^3} \| u \|_{H^2} \\
&\qquad \quad {} + |\th|^2 \| u\|_{H^4} \| u \|_{H^3} \| u\|_{H^2}^3\big)\\
& \quad {} +\tfrac{t^2}{\eps^2} \big( |\th|\,\| U \, u\|_{L^2} \| u \|_{H^4} \| u \|_{H^3} + |\th|^2 \| u \|_{H^4} \| u \|_{H^3} \| u \|_{H^2}^3 + |\th|^2 \| U\, u\|_{L^2} \| u \|_{H^2}^4 \\
& \qquad \quad {} + |\th|^3 \| u \|_{H^2}^7\big) +\bcor{\Order\big( \tfrac{t^3}{\eps^3} + \tfrac{t^3}{\eps} + t^3\,\eps\big)}\,,
\end{align*}
and \begin{align*}
\| T_2\|_{L^2} & \leq \eps\big(  \| u \|_{H^2} + |\th|\,\| u \|_{H^4} \| u \|_{H^3} \| u \|_{H^2}\big) \\
& \quad {} + t\, \big( \| U^2 u \|_{L^2} + \| U \, u \|_{L^2} + \| u \|_{H^4} + |\th|\,\| U\, u\|_{L^2}\| u \|_{H^4} \| u \|_{H^2} \\
& \qquad {} + |\th|\,\| u \|_{H^4} \| u \|_{H^2}^2+ |\th|^2 \| u\|_{H^4} \| u \|_{H^3} \| u\|_{H^2}^3 \big)\\
& \quad {} + \bcor{\Order\big(\tfrac{t^2}{\eps}+ t^2\, \eps\big)}\,.
\end{align*}
Concerning $T_{22}$ we refer to the representation of the commutator $ [B,A](v) $ in Section~\ref{commut}.

Combining these results, we obtain
\begin{align*}
&\| \nP(t,u) - \nL(t,u) \|_{L^2}\leq C_1\,t^3\,
    e^{\big( C_2\,\tfrac{t}{\eps}\,|\th|\,\sup_{0 \leq \chi \leq \sig \leq t}
    \|\EF(\sig,\nS(\chi,u)) \|_{H^2}^2 \big)}  {}\\
&\qquad \cdot{}\Big( \eps\,\big(\| u \|_{H^2} + |\th|\,\| u \|_{H^4} \| u \|_{H^3} \| u \|_{H^2}\big)\\
& \qquad \quad {} + \tfrac{1}{\eps}\big(  \| U \, u \|_{L^2} + |\th|\,\| U\, u\|_{L^2} \| u \|_{H^4} \| u \|_{H^2} + |\th|\,\|u\|_{H^2}^3 + |\th|^2 \| u \|_{H^2}^5 + C_\ast\big) \\
& \qquad \quad {} + t \big(  \| U^2 u\|_{L^2}+ \| U \, u \|_{L^2} + \| u \|_{H^4} + |\th|\,\| u \|_{H^3}\| u \|_{H^2}^2 + |\th|\,\| U \, u \|_{L^2} \| u \|_{H^3} \| u \|_{H^2} \\
& \qquad \qquad \quad {} + |\th|^2 \| u\|_{H^4} \| u \|_{H^3} \| u\|_{H^2}^3\big) + \tfrac{t}{\eps}\bcor{C_\ast}\\
& \qquad \quad {} + \tfrac{t}{\eps^2} \big( |\th|\,\| U \, u\|_{L^2} \| u \|_{H^4} \| u \|_{H^3} + |\th|^2 \| u \|_{H^4} \| u \|_{H^3} \| u \|_{H^2}^3 + |\th|^2 \| U\, u\|_{L^2} \| u \|_{H^2}^4 \\
& \qquad \qquad \qquad {} + |\th|^3 \| u \|_{H^2}^7\big)\Big)\\
& \qquad {} + \bcor{\Order\big( \tfrac{t^5}{\eps^3}+ \tfrac{t^5}{\eps}+ t^5\, \eps\big)}\,,
\end{align*}
\bcor{where the constant $C_\ast$ arising in Appendix~\ref{C1} appears again.}
\bcor{To sum up, for the Lie splitting method
\begin{align}
\label{eq:apost-Lie}
\| \nP_{\text{Lie}}(t,u)- \nL_{\text{Lie}}(t,u)\|_{L^2}
&{} \lesssim t^3\,\big(C_1\,\tfrac{1}{\eps} + C_2\,\eps \big) + t^4 \big(C_3\,\tfrac{1}{\eps^2} + C_4 \big)\,,
\end{align}
with constants $C_1$,\ldots,$C_4$ depending on the $H^4$\,-\,norm of $u$,
and additionally $C_1,\,C_3$ depending on $\|U\, u\|_{L^2}$, and $C_4$ depending on $\|U^2\, u\|_{L^2}$.}

This estimate is of a similar nature as the a~priori error bound for the Strang splitting method.
Therefore we expect the same asymptotical behavior of both the deviation of the a~posteriori error estimator
of the Lie splitting method and the local error of the Strang splitting method.
This claim is mainly based on the fact that both estimates are dominated by the commutators
$\|[[B,A],A](u) \|_{L^2}$ and $\| [[B,A],B](u)\|_{L^2}$.

\paragraph{Strang splitting.}
It was observed earlier that, to leading order, the deviation of the a~posteriori Lie error estimator
is dominated by $\|[[B,A],A](u) \|_{L^2}$ and $\| [[B,A],B](u)\|_{L^2}$.
For the case of Strang splitting, third-order commutators dominate. According to~\cite{Auz14},
\begin{equation}
\label{eq:err-est-Strang}
\nP(t,u) - \nL(t,u)
= \int_{0}^{t} K_3(\t,t)\,\pdtau{}
                 \big(\pdd\,\nE_F(t-\t,\nS(\t,u))\,\S3(\t,u) \big)\,\dd \t + \Order(t^5)\,,
\end{equation}
with the third-order Peano kernel $ K_3 (\t,t) = \tfrac{1}{6}\,\t\,(t-\t)^2 $
associated with the error of the underlying
Hermite quadrature rule defining $ \nP $, and the third-order defect
\begin{align*}
\S3(t,u) &= \Big\{ \EA(\half\,t)\,\pdd\,\EB(t,v)
                   \big( \half\,[[B,A],A](v) + \half\,[[B,A],B](v) \big) \\
& \qquad {} - \EA(\half\,t)\,\big( \tfrac{3}{4}\,[[B,A],A](w) + [[B,A],B](w) \big)
            + \nO(t)\Big\}~\Big|_{\,\begin{subarray}~~v = \EA(\shalf t)\,u  \\
                                                     ~w = \EB(t,\EA(\shalf t)\,u) \end{subarray}}\,.
\end{align*}
We now collect coefficients of $t^4$ in~\eqref{eq:err-est-Strang},
\begin{align*}
\nP(t,u) - \nL(t,u)
&= C\,t^4 \Big(\tfrac{1}{4}\,[A,[B,[B,A]]](u) + \tfrac{1}{4}\,[A,[A,[B,A]]](u) - [B,[B,[B,A]]](u) \\
& {} -\half\,[A,[B,[B,A]]](u) - \tfrac{3}{4}\,[B,[A,[B,A]]](u)
			 - \tfrac{3}{8}\,[A,[B,[B,A]]](u) \Big) + \Order(t^5)\\
&= C\,t^4 \Big(-\tfrac{1}{4}\,[A,[B,[B,A]]](u) - \tfrac{1}{8}\,[A,[A,[B,A]]](u) \Big) + \Order(t^5)\,.
\end{align*}
Here we have used the identity $[A,[B,[B,A]]](u) = [B,[A,[B,A]]](u)$ and the fact that,
for a cubic nonlinearity and a harmonic potential, $[B,[B,[B,A]]]=0$.\\

\bcor{The appearing commutators can be estimated as
$\| [A,[B,[B,A]]](u) \|_{L^2} \leq \eps^0\,C\,(\| u \|_{H^4})$
and $\| [A,[A,[B,A]]](u)\|_{L^2}\leq \eps^2\,C\,(\| u \|_{H^6})$, such that
\begin{align}
\label{eq:apost-Strang}
\| \nP_{\text{Strang}}(t,u)- \nL_{\text{Strang}}(t,u)\|_{L^2}
&{} \lesssim C\, t^4\,\big(1  + \eps^2 \big)  \,,
\end{align}
with $C$ depending in particular on $\| u \|_{H^6}$.}
\bcor{From~\eqref{eq:err-est-Strang} we deduce that the coefficients hidden in the $\Order(t^5)$ remainder are of the form $\bcor{\Order(\frac{t^5}{\eps}+ t^5\, \eps)}$, which is also observed numerically,
see Fig.~\ref{Fig:h_Aposteriori_6methods}.}

\section{Numerical results} \label{sec:numerics}
In our numerical experiments we solve a cubic Schr{\"o}dinger equation~\eqref{eq:Problemcomplete},
with either no potential $U=0$ or a quadratic potential
$U:\,\RR^d \rightarrow \RR$,\, $x \mapsto \half |x|^2$.
For $\th=1$ we speak of a defocussing nonlinearity, for $\th=-1$ we have a focussing nonlinearity.

For the computation of reference solutions
we have used a fourth-order scheme with 7~stages and a sixth-order scheme
with 11~stages from~\cite{Bla02}.
Besides these two and the Lie and Strang splitting methods we have tested a third-order scheme from~\cite{Ruth83},
a fourth-order scheme from~\cite{Yosh90},
and a new fifth-order scheme (see Table~\ref{tab:Auz5}).
\begin{table}[h!]
\caption{\label{tab:Auz5}Coefficients of a new 5-th order scheme obtained on the basis of order conditions
set up according to~\cite{auzinger2014local}.}
\begin{center}
		\begin{tabular}{| l | r || l |r|}
			\hline
	        $a_1^{\phantom{A}}$ & $0.475018345144539497$ & $b_1$ & $\,-\,0.402020995028838599$\\
	        $a_2$ & $0.021856594741098449$ & $b_2$ & $0.345821780864741783$\\
	        $a_3$ & $\,-\,0.334948298035883491$ & $b_3$ & $0.400962967485371350$\\
	        $a_4$ & $0.512638174652696736$ & $b_4$ & $0.980926531879316517$\\
	        $a_5$ & $\,-\,0.011978701020553904$ & $b_5$ & $\,-\,1.362064898669775625$\\
	        $a_6$ & $\,-\,0.032120004263046859$ & $b_6$ & $0.923805029000837468$\\
	        $a_7$ & $0.369533888781149572$ & $b_7$ & $0.112569584468347105$	\\
	        \hline	
		\end{tabular}
\end{center}
\end{table}

\subsection{Numerical order estimation}
\label{sec:num-order}
We compare the numerically observed local error behavior for different choices of $\eps$ as well as for a time stepsize $t$ \bcor{proportional to} this parameter ($t=\eps$).

For $ \eps = 1 $, we observe the classical order $\Order(t^{p+1})$.
For \ncor{small $\eps \approx 10^{-2}$}, a different local error behavior is in fact observed. \bcor{We also have noticed distinct asymptotics in dependence of the smoothness of the initial value.}
\paragraph{Smooth initial state.}
We may express \bcor{the} observed dependencies as
\begin{equation}
\|\nL(t,u)\|_{L^2} \approx t^{p+1} \cdot
\begin{cases}
\,C\,\big( 1 + \tfrac{t}{\eps} \big),  & p~\,\text{odd}\,, \\
\,C\,\tfrac{1}{\eps}\,,                & p~\,\text{even}\,.
\end{cases}
\end{equation}
\bcor{The dependence on $\tfrac{t}{\eps}$ for odd order methods is visible
in Fig.~\ref{Fig:h_Aposteriori_6methods} as a kink, while the dependence on $\tfrac{1}{\eps}$ for even order methods appears as an order reduction in Fig.~\ref{Fig:h-e_Aposteriori_6methods} for the specific choice $t=\eps$.}
\bcor{This reflects the theoretical results in~\eqref{eq:LocErr} and~\eqref{eq:LocErrLie} for smooth initial values,
\begin{align}
\label{eq:Lie_simplified}
\big\| \nL_{\text{Lie}} (t,u) \big\|_{L^2}
&\lesssim C\,\big(t^2 +  t^3\,(\tfrac{1}{\eps} + \eps) + t^4\,(\tfrac{1}{\eps^2} + 1) \big) \,, \\
\intertext{with $C$ depending on the $H^2$\,-\,norm of $u$ and on $\|U\, u\|_{L^2}$, and for the Strang splitting,}
\big\| \nL_{\text{Strang}} (t,u) \big\|_{L^2}
&\lesssim t^3\,\big(C_1\,\tfrac{1}{\eps} + C_2\,\eps \big) + t^4\,\big(C_3\,\tfrac{1}{\eps^2} + C_4 \big) \,,
\end{align}
with constants depending on the $H^4$\,-\,norm of $u$, as well as $C_1$, $C_3$ depending on $\|U\, u\|_{L^2}$, and $C_4$ depending on $\|U^2\, u\|_{L^2}$.}
%
\bcor{\paragraph{WKB initial state.} Oscillatory initial data given in WKB form leads to less regular solutions, which reveals a different aspect of the theoretical estimates.}

\bcor{For our numerical experiments we chose
\begin{align}
\label{eq:WKB}
\psi(x,y,0) = \ee^{-x^2}\cdot \ee^{-\ii/\eps\big(\log(\exp(x)+\exp(-x))\big)}\,,
\end{align}
which features oscillations in dependence of $\eps$.}

\bcor{The numerical observations in Fig.~\ref{Fig:WKB-figures} yield
\begin{align}
\|\nL(t,u)\|_{L^2} \approx C\, t^{p+1} \,\tfrac{1}{\eps}\,,
\end{align}
also in accordance with~\cite{descombes2013lie}. The theoretical results~\eqref{eq:LocErr} and~\eqref{eq:LocErrLie} would imply too pessimistic estimates, since higher powers of $\frac{1}\eps$ are introduced by the estimates in Sobolev spaces. In particular for this situation where the
time stepsize would be underestimated a priori, the use of automatic stepsize control is indicated.}

\bcor{\paragraph{Global error observations.} The influence of higher order terms in the estimates is observed  in numerical experiments more distinctly in the global error. Thus in Fig.~\ref{Fig:Global-err}, we observe that for small $\eps$, the effects for the local error may sum up critically in dependence of $\eps$. For $t\lesssim \eps<1$, the classical global order $\Order(t^{p})$ is observed, for $\eps \lesssim t<1$, the contributions of terms involving higher powers of $\frac{t}{\eps}$ imply stagnation at a constant value (see Fig.~\ref{Fig:Global-err}).}

  \begin{figure}[h!]
   	\centering
       \includegraphics[width=0.48\textwidth]{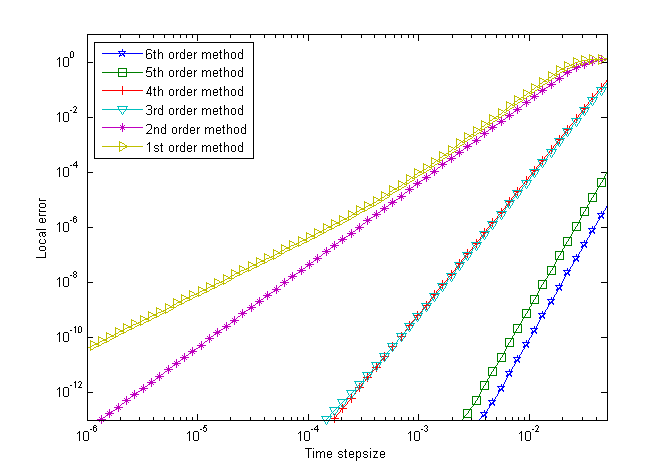}
       \includegraphics[width=0.48\textwidth]{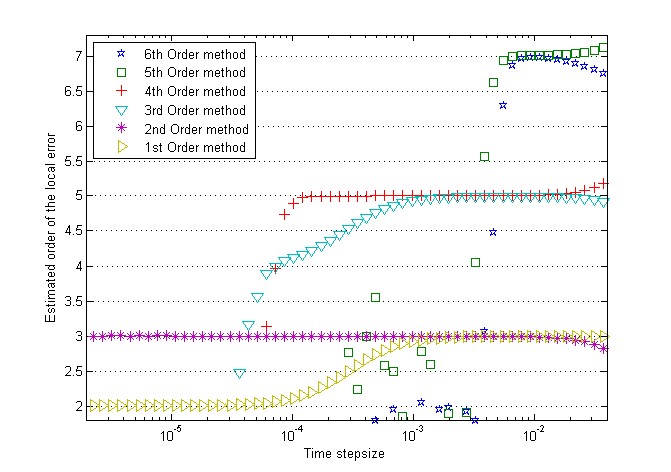}\\  	
       \includegraphics[width=0.48\textwidth]{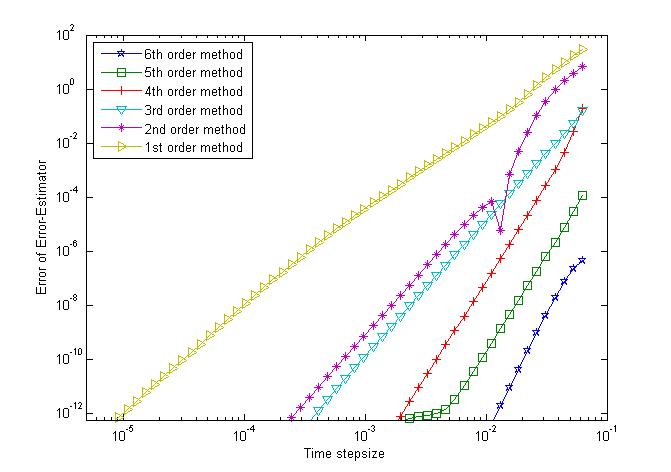}
       \includegraphics[width=0.48\textwidth]{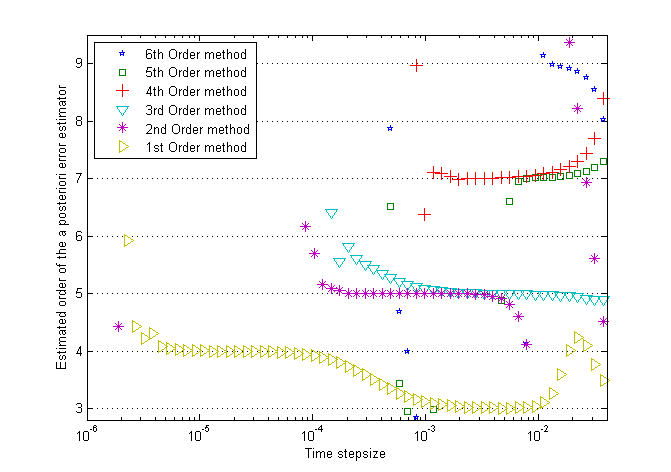}
   	\caption{\textbf{$t$\,-\,dependence of the local error and of the deviation of the a~posteriori error estimator.}
             \textbf{First row:} The plot on the left shows the empirical local error
             for several splitting methods, while the plot on the right shows the associated observed orders.
             It can be seen that for the Lie splitting \bcor{and other odd-order} methods,  the order decreases by one,
             starting below $t\approx \eps$.
             The even order methods are not affected by this order reduction.
    \textbf{Second row:} Here, the plot on the left shows the estimated deviation of the
             a~posteriori error estimator for several splitting methods and again,
             the plot on the right shows the associated observed orders.
             The odd order methods change their behavior,
             but here we observe an improved order for $t < \eps$.
             Moreover, the even order methods perform even one order better than expected,
             $\Order(t^{p+3})$.
             For all computations the initial condition was a shifted Gaussian at
             $2\cdot10^4$ gridpoints with a fixed parameter $\eps=\bcor{10^{-2}}$.}
   	\label{Fig:h_Aposteriori_6methods}
   \end{figure}

  \begin{figure}[h!]
   	\centering
       \includegraphics[width=0.48\textwidth]{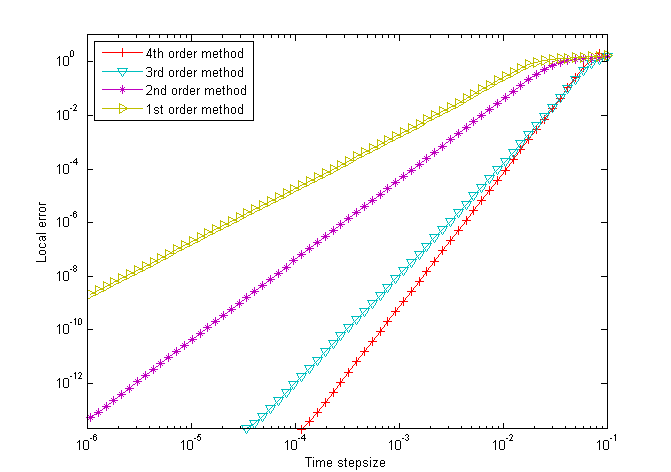}
       \includegraphics[width=0.48\textwidth]{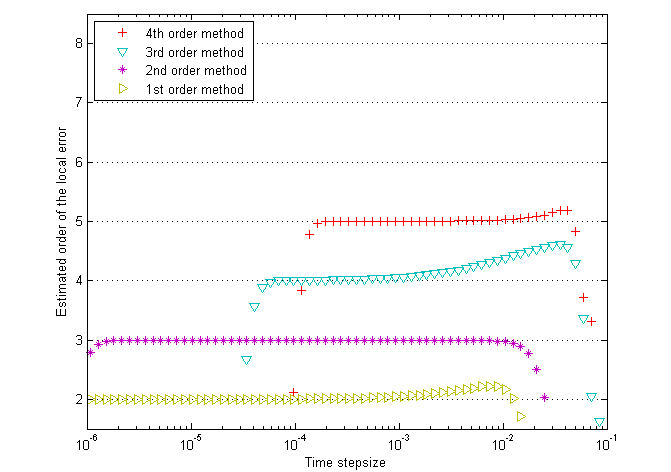}\\  	
       \includegraphics[width=0.48\textwidth]{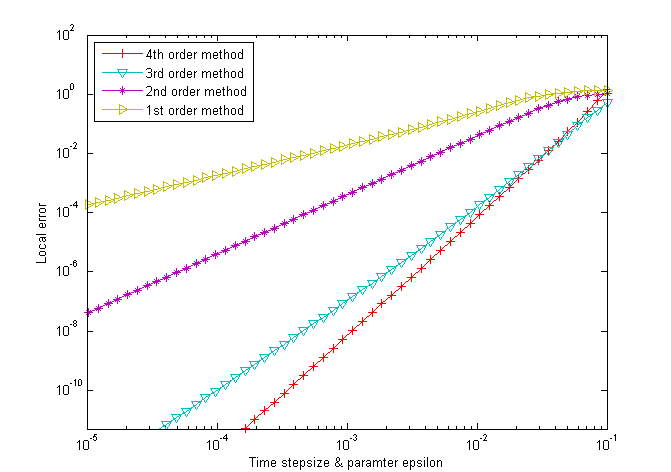}
       \includegraphics[width=0.48\textwidth]{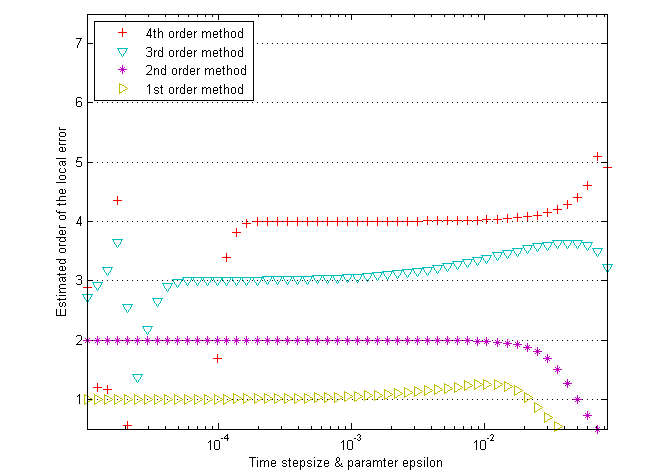}
   	\caption{\bcor{\textbf{Dependence of the local error for WKB initial values.}
             \textbf{First row:} $t$\,-\,dependence. The plot on the left shows the empirical local error
             for several splitting methods, while the plot on the right shows the associated observed orders.
             For all methods the order $\Order(t^{p+1})$ can be observed (We have chosen $\eps=10^{-2}$ and $2\cdot10^4$ gridpoints.).
    \textbf{Second row:} $(t=\eps)$\,-\,dependence. The plot on the left shows the empirical local error
                 for several splitting methods, while the plot on the right shows the associated observed orders.
                 It can be seen that for all methods the order decreases by one,
                 starting below $t\approx \eps$.
                 To resolve the error also for $\eps=10^{-5}$, $4\cdot10^5$ gridpoints have been used.
             For all computations the initial condition is given in~\eqref{eq:WKB}.}}
   	\label{Fig:WKB-figures}
   \end{figure}

\begin{figure}[h!]
   	\centering
       \includegraphics[width=0.48\textwidth]{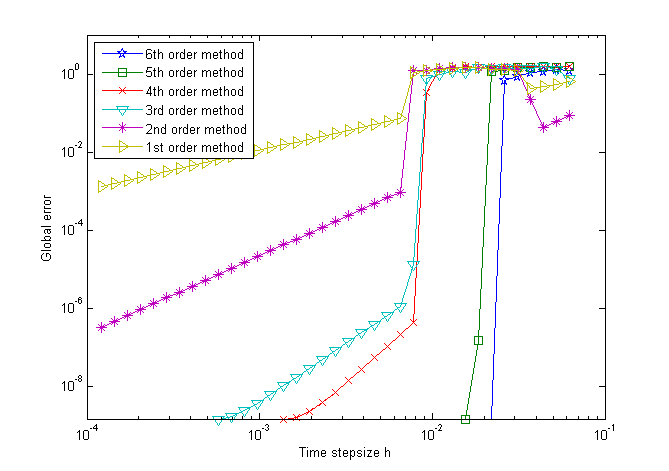}
       \includegraphics[width=0.48\textwidth]{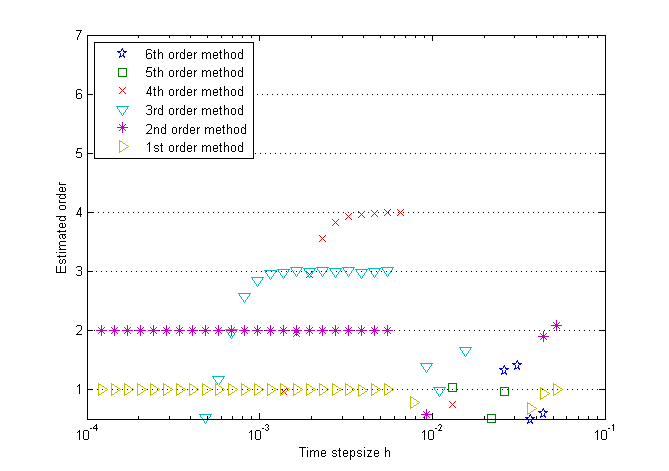}
   	\caption{\bcor{
                   \textbf{$t$-dependence of the global error for Gaussian initial values.} For
                   $\eps=1/250$, the error initially stagnates for $t\gtrsim \eps$ and resumes the
                   classical order for smaller $t\lesssim \eps$, as likewise observed for WKB initial values. A total integration time of $T=0.5$ and $2\cdot 10^4$ gridpoints in space have been used.}}
   	\label{Fig:Global-err}
   \end{figure}

\subsection{Observed behavior of the deviation of the error estimator}
For $\eps\approx 1$ the deviation of the local error estimator~~\eqref{eq:ErrEst} associated
with a method of order~$p$ shows an $\Order(t^{p+2})$ behavior,
but \ncor{for smaller values of $\eps$,}
we have observed a different behavior for some methods, as shown in
Figs.~\ref{Fig:h_Aposteriori_6methods} and ~\ref{Fig:h-e_Aposteriori_6methods}.

In detail, the following dependencies have been observed for smooth initial values independent of $\eps$,
 \begin{equation}
 \label{eq:aposterioridependence}
 \|\nP(t,u) - \nL(t,u)\|_{L^2} \approx t^{p+2} \cdot \begin{cases}
 \,C\,\frac{1}{\eps}\,\frac{t}{\eps + t}\,, & p~\,\text{odd,} \\
 \,C\,\frac{1}{\eps}\,t\,,                  & p~\,\text{even.}
 \end{cases}
 \end{equation}
\bcor{The functional form $\frac{t}{\eps +t}$ above is inferred from the kink observed in the empirical convergence order in Fig.~\ref{Fig:h_Aposteriori_6methods}. The increased order $\Order(t^{p+3})$ for even order methods is present in Fig.~\ref{Fig:h_Aposteriori_6methods} and the factor $\frac{1}{\eps}$ is apparent from Fig.~\ref{Fig:h-e_Aposteriori_6methods}.}

\bcor{The theoretical result~\eqref{eq:apost-Lie} shows a deviation of the a posteriori estimator of order $\nO(t^{p+2})$ for the Lie splitting method. In contrast, an increased order in our numerical experiments occurs in the same regime as the order reduction for the local error behavior, showing a spurious improvement to order $\nO(t^{p+3})$.}

\bcor{For the choice $t=\eps$ our observations in Fig.~\ref{Fig:h-e_Aposteriori_6methods} reflect the analytical results~\eqref{eq:apost-Lie} and~\eqref{eq:apost-Strang} ($\Order(t^{p+1})$ for the Lie splitting method, $\Order(t^{p+2})$ for the Strang splitting method).}

\bcor{Numerical experiments, not reported here, show that for WKB initial values the observed order reduction analogous to Section~\ref{sec:num-order} is
\[ \|\nP(t,u) - \nL(t,u)\|_{L^2}\approx C\, t^{p+2} \,\tfrac{1}{\eps}\,.\]
Again, the theoretical estimates~\eqref{eq:apost-Lie} and~\eqref{eq:apost-Strang} are too pessimistic for this case.}


  \begin{figure}[h!]
    	\centering
        \includegraphics[width=0.48\textwidth]{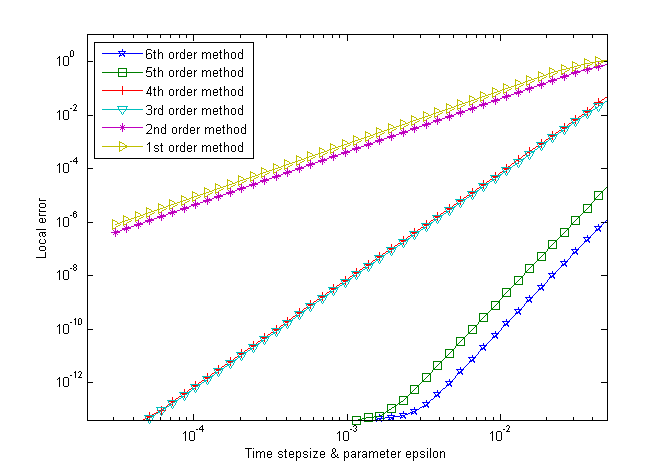}
        \includegraphics[width=0.48\textwidth]{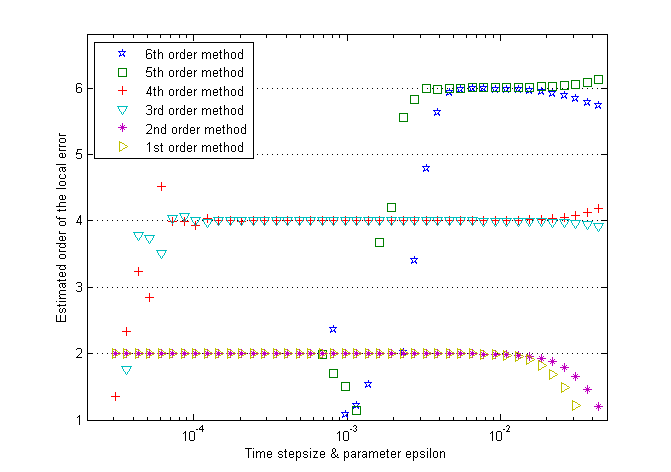}\\
        \includegraphics[width=0.48\textwidth]{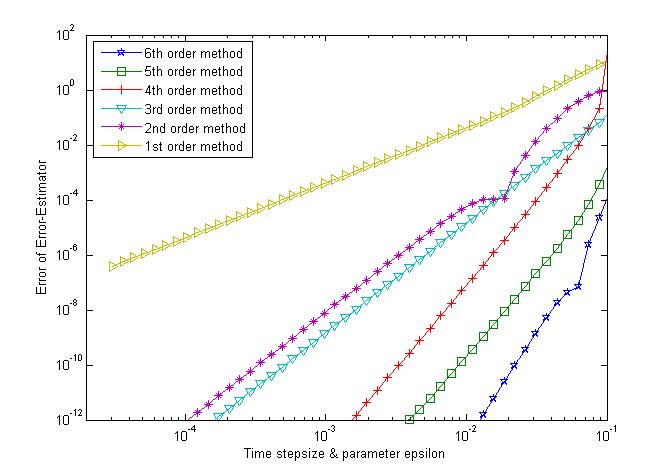}
        \includegraphics[width=0.48\textwidth]{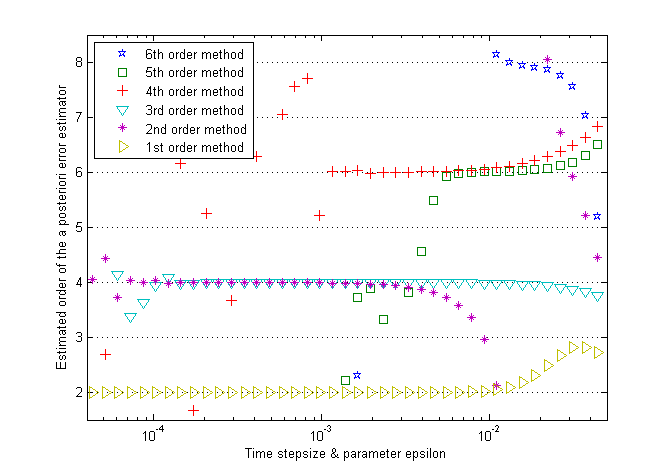}
    	\caption{\textbf{
    $(t=\eps)$\,-\,dependence of the local error and of the deviation of the a~posteriori error estimator.} \textbf{First row:} As in Fig.~\ref{Fig:h_Aposteriori_6methods}, the plot on the left shows the empirical error of different splitting methods, while the plot on the right shows their observed orders.
    It is obvious that here the local order of the even order methods is reduced to $\Order(t^p)$ while the odd order methods are not affected. \textbf{Second row:} Again, the plot on the left shows the empirical deviation of the a~posteriori error estimator for different splitting methods and the plot on the right shows their observed orders. Compared to Fig.~\ref{Fig:h_Aposteriori_6methods}, the odd order methods suffer from an order reduction, while the even order methods show the expected dependence $\Order(t^{p+2})$.
    This advantage can be used to overcome the disadvantage observed in the first row (see~\eqref{eq:aposterioridependence}) even for less regular initial conditions.
    For all computations the initial condition was a shifted Gaussian at $2\cdot10^4$ gridpoints.}
    \label{Fig:h-e_Aposteriori_6methods}
    \end{figure}

\subsection{Defocussing laser beams and soliton solutions: Adaptive integration}
For the following experiment we choose an application where a cubic Schr{\"o}dinger equation
without external potential arises, namely a model involving a self-defocussing laser beam in a nonlinear medium (see~\cite{McDon93}).

The model describes the propagation of a weak intensity beam $\psi(x,y,z)$ in $z$\,-\,direction via
\begin{equation}
\label{eq:laserbeam}
\ii\,\eps \ncor{\tfrac{\partial}{\partial z} \psi(x,y,z) = -\tfrac{1}{2}\,\eps^2\, \big( \tfrac{\partial^2}{\partial x^2} + \tfrac{\partial^2}{\partial y^2}\big) \psi(x,y,z)
 + \vartheta | \psi(x,y,z) |^2\,\psi(x,y,z)\,,}
\end{equation}
where $\eps$ describes the relationship between diffusion and focussing (arising from the nonlinear medium).
For the special initial distribution $\psi(x,y,0)=\tanh(x)$ and $\eps=1$, $\vartheta=1$,
we obtain the solution
\begin{equation*}
\psi(x,y,z) = \tanh(x)\,\ee^{-\ii\,t\,z}\,.
\end{equation*}
We have modified this and constructed a wave similar to a soliton by multiplying a Gaussian by $\tanh$,
which might be more stable under diffusion.
For the results shown in Fig.~\ref{Fig:SoliGauss} we have used the initial conditions
\begin{align*}
\psi_1(x,y,0) &= A\,\exp\big(-\tfrac{x^2+y^2}{r_0^2}\big)\,\tanh\big(\tfrac{y}{y_s}\big)\,,\\
\psi_2(x,y,0) &= A\,\exp\big(-\tfrac{x^2+y^2}{r_0^2}\big)\,.
\end{align*}
Numerical solutions have been obtained at $1000$ spatial gridpoints on the $x$\,- and $y$\,-\,axes and
by a time-adaptive method of order four based
on the a~posteriori local error estimator~\eqref{eq:ErrEst} with a local absolute tolerance~$10^{-8}$.
Comparing the two columns, we indeed see that the $\tanh$ profile
provides a more stable signal than the Gaussian, which diffuses much faster and shows higher oscillations.

\begin{figure}[h!]
	\centering
    \includegraphics[width=0.48\textwidth]{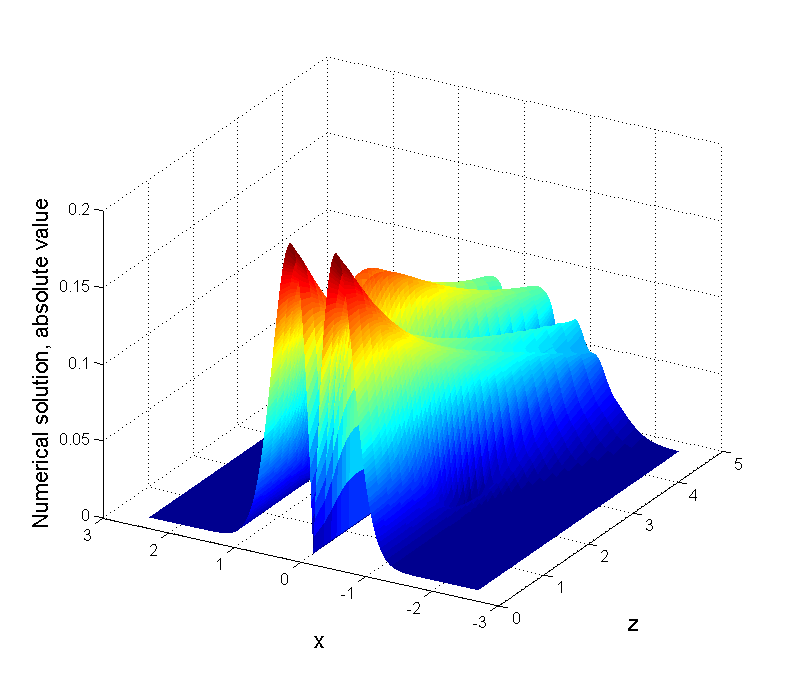}
    \includegraphics[width=0.48\textwidth]{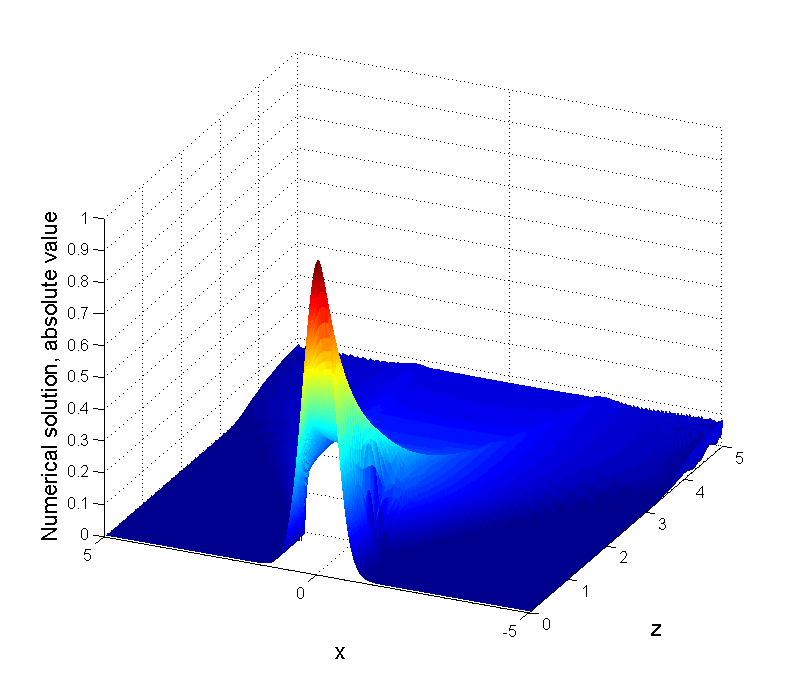}\\
    \includegraphics[width=0.48\textwidth]{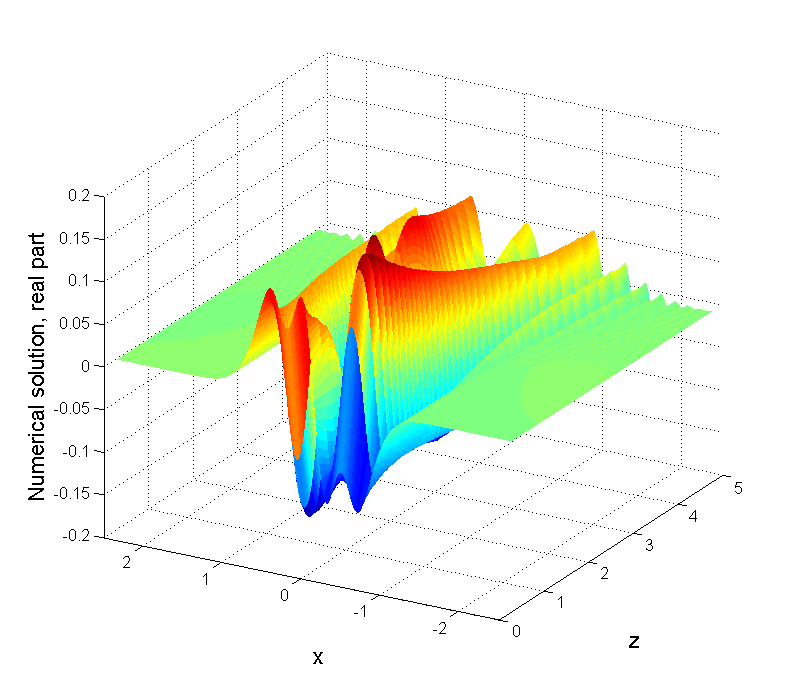}
    \includegraphics[width=0.48\textwidth]{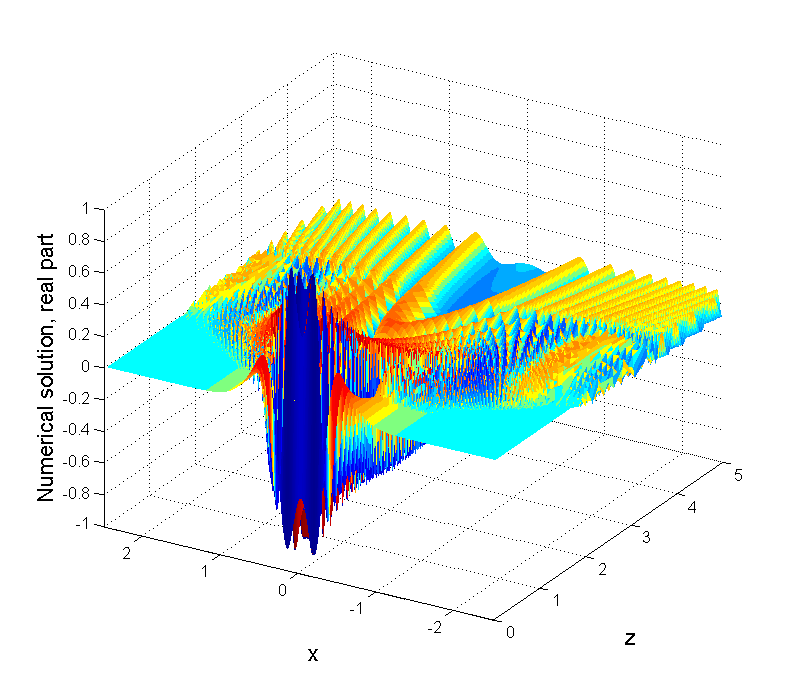}\\
    \includegraphics[width=0.48\textwidth]{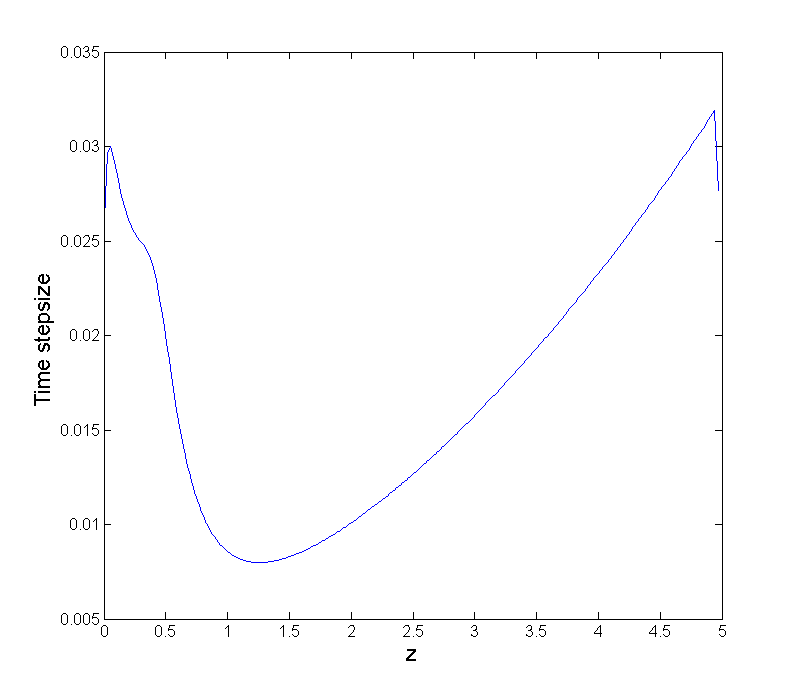}
    \includegraphics[width=0.48\textwidth]{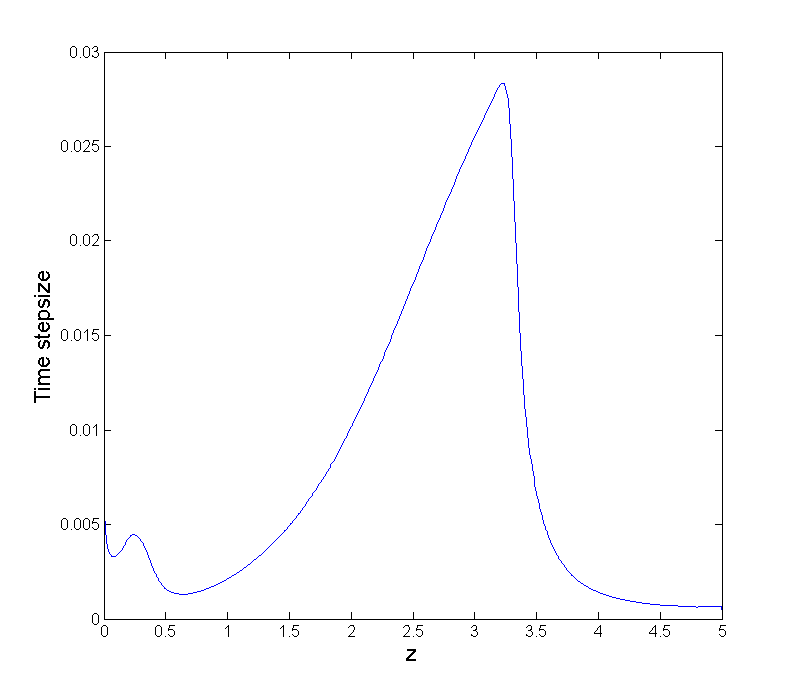}
	\caption{\textbf{Defocussing laser beam with different shapes.}
             Absolute value of beam intensity $\psi$ for problem~\eqref{eq:laserbeam} with $\eps=\tfrac{1}{100}$, $T_{\text{end}}=5$ along the $x$\,- and $z$\,-\,axes for $y=0$.
             In the first two rows we compare the absolute value (first row) and the real part (second row) of the solution. In the third row we display the adaptive stepsizes.
             \textbf{Left column:}
             Modulated Gaussian
             $\psi_1(x,y,0)=\exp(-4(x^2+y^2))\,\tanh(x)$.
             \textbf{Right column:} Gaussian $\psi_2(x,y,0)=\exp(-4(x^2+y^2))$.}
	\label{Fig:SoliGauss}
\end{figure}

\section*{\large Appendix}
\appendix
\section{Technical tools}
\label{Technical Tools}

\subsection{Gr{\"o}bner-Alekseev formula (nonlinear variation of constant)}
\begin{proposition}
\label{Grobner}
Given a pair of initial value problems,
\begin{subequations}
\label{eq:GA}
\begin{align}
&\begin{cases}
~z'(t) = G(t,z(t)) = F(z(t)) + r(t,z(t))\,,\qquad 0 \leq t \leq T \label{eq:GA-1} \\
~z(0) = u
\end{cases} \\
&\begin{cases}
~y'(t) = F(y(t))\,,\qquad \qquad \qquad \qquad \qquad  \quad~~  0 \leq t \leq T \label{eq:GA-2} \\
~y(0) = u
\end{cases}
\end{align}
\end{subequations}
the solution $z(t)$ of~\eqref{eq:GA-1} can be expressed
via the nonlinear variation of constant formula
\begin{equation}
\label{eq:NLVOC}
z(t) = \nE_G(t,u) =
y(t) + \int_{0}^{t} \pdd\,\nE_F(t-\t,\nE_G(\t,u))\,r(\t,\nE_G(\t,u))\,\dd\t,\quad 0 \leq t \leq T\,,
\end{equation}
where $y(t)=\nE_F(t,u)$ is the solution of~\eqref{eq:GA-2}.
\end{proposition}
\begin{proof}
See~\cite{Hair93} or~\cite{descombes2013lie}.
\end{proof}
\subsection{$L^2$ estimates for products of functions} \label{app:A2}
For the following estimates we use H\"{o}lder's inequality and
Sobolev embeddings for estimating products in the $L^2$\,-\,norm,
for spatial dimension $d\in \{1,2,3\}$,
\begin{align*}
\| u\,v \|_{L^2} &\leq \| u \|_{L^4}\,\| v \|_{L^4} \leq C\,\| u \|_{H^1}\,\| v \|_{H^1}\,, \\
\| u\,v \|_{L^2} &\leq  C\,\| u \|_{H^2}\,\| v \|_{L^2}\,, \\
\| u\,v\,w\|_{L^2} &\leq  \| u \|_{L^6}\,\| v \|_{L^6}\,\| w \|_{L^6}
                    \leq C\,\| u \|_{H^1}\,\| v \|_{H^1}\,\| w \|_{H^1}\,, \\
\| u\,v\,w \|_{L^2} &\leq  C\,\| u\,v \|_{H^2}\,\| w \|_{L^2}\,, \\
\| u\,v\,w\,z \|_{L^2} &\leq  C\,\| u\,v\,w\|_{H^2}\,\| z \|_{L^2}\,.
\end{align*}
Since $H^2$ forms an algebra, we can moreover estimate products in $H^2$ as
\begin{equation*}
\| u\,v \|_{H^2}\leq C\,\| u \|_{H^2}\,\| v \|_{H^2}\,.
\end{equation*}

\subsection{$L^2$ estimates for mixed powers of $x$ and $\partial_j u$} \label{app:A3}
The following estimates are based on repeated \ncor{integration by parts} and the inequality of arithmetic and geometric means.
\begin{align*}
\| x_j\, \partial_j\, u\|_{L^2} & \leq C\, \big( \| x^2\, u\|_{L^2} + \| u\|_{H^2}\big)\,,\\
\| x_j^2\, \partial_j\, u\|_{L^2} & \leq C\, \big( \|x^4\, u \|_{L^2}+ \|x^2\, u\|_{L^2} + \|u\|_{H^2}\big)\,,\\
\| x_j\, \partial_j^2\, u \|_{L^2}& \leq C\, \big( \|x^3\, u \|_{L^2} + \|u\|_{H^3}\big),
  \quad \text{or} \quad  \| x_j\, \partial_j^2\, u \|_{L^2} \leq C\, \big( \|x^3\, u \|_{L^2} + \| u \|_{H^3}\big)\,,\\
\bcor{\| x_j^2\, \partial_j^2\, u \|_{L^2}}& \leq C\, \big(\|x^4\, u \|_{L^2} + \|u\|_{H^4}\big)\,.
\end{align*}

\section{Derivation of integral representations for the first- and second-order defect terms} \label{lab:expansiondet}
The integral representations~\eqref{eq:S1integral} and~\eqref{eq:S2integral}
for the first- and second-order defect terms $ \S1(t,u) $ and $ \S2(t,u) $ are related to the analogous
results for the general case, with $ A $ and $ B $ nonlinear, from~\cite{Auz14},
which were derived and verified with the help of computer algebra.
Here we specialize for $ A $ linear and give a rigorous proof,
rearranging terms in a way which is appropriate for the present purpose,
without explicating all technical details.

The idea is to evaluate the defect terms in a way containing no explicit time derivatives.
This results in several subexpressions vanishing at $ t=0 $
and satisfying certain linear evolution equations.
Application of the variation of constant formulas
\begin{subequations}
\label{eq:VOCLIN}
\begin{align}
y'(t) &= \half\,A\,y(t) + r(t)\,, ~~ y(0)=0
\quad \Rightarrow \quad
y(t) 
     = \int_{0}^{t} \EA(\half\,(t-\t))\,r(\t)\,\dd\t\,,\label{eq:VOCLIN-1}  \\
y'(t) &= B'(\EB(t,u))\,y(t) + r(t),\, ~~y(0)=0
\quad \Rightarrow \quad
y(t) =  \int_0^t \pdd\,\EB(t-\t,\EB(\t,u))\,r(\t)\,\dd\t\,. \label{eq:VOCLIN-2}
\end{align}
\end{subequations}
then yields the desired integral forms.

To explain~\eqref{eq:VOCLIN-2} we note that
\begin{align*}
\begin{aligned}
\pdt{}\pdd\EB(t,u)z &= B'(\EB(t,u))\,\pdd\EB(t,u)z \,,\\
\big(\pdd\EB(t,u)\big)^{-1} &= \pdd\EB(-\t,\EB(\t,u)) \,,\\
\pdd\EB(t,u) \pdd\EB(-\t,\EB(\t,u)) &= \pdd\EB(t-\t,\EB(\t,u))\,.
\end{aligned}
\end{align*}
Hence by the linear variation of constant formula,
\begin{align*}
y(t) &= \pdd \EB(t,u) \int_0^t \big( \pdd \EB(\t,u)\big)^{-1} r(\t)\, \dd \t = \int_0^t \pdd\,\EB(t-\t,\EB(\t,u))\,r(\t)\,\dd\t\,.
\end{align*}

\subsection{The first-order defect $ \S1(t,u) $}\label{sec:S1herleitung}
The intermediate values of a Strang splitting step~\eqref{eq:Strang},
\begin{subequations}
\label{eq:wv_def}
\begin{align}
v &= v(t,u) = \EA(\half\,t)\,u\,, \\
w &= w(t,u) = \EB(t,v)\,,
\end{align}
\end{subequations}
(such that $ \nS(t,u) = \EA(\half\,t,w) $) satisfy
\begin{align*}
\pdt{} v &= \half\,A\,v\,, \\
\pdt{} w &= B(w) + \pdd\,\EB(t,v)\,\pdt{} v = B(w) + \half\,\pdd\,\EB(t,v)\,A\,v\,.
\end{align*}
Thus,
\begin{align*}
\pdt{} \nS(t,u)
&= \pdt{} \big(\EA(\half\,t)\,w \big) \\
&= \half\,A\,\nS(t,u) + \EA(\half\,t,w)\,\pdt{} w \\
&= \half\,A\,\nS(t,u) + \EA(\half\,t)\,B(w) + \half\,\EA(\half\,t)\,\pdd\,\EB(t,v)\,A\,v\,,
\end{align*}
and for the defect $ \S1 $ this gives
\begin{align*}  
\S1(t,u)
&= \pdt{} \nS(t,u) - A\,\nS(t,u) - B(\nS(t,u)) \\
&= \EA(\half\,t)\,B(w) + \half\,\EA(\half\,t)\,\pdd\,\EB(t,v)\,A\,v
   - \half\,A\,\nS(t,u) - B(\nS(t,u)) \\
&= \EA(\half\,t)\,B(w) + \half\,\EA(\half\,t)\,\pdd\,\EB(t,v)\,A\,v
   - \half\,A\,\EA(\half\,t)\,w - B(\EA(\half\,t)\,w) \\
&= \EA(\half\,t)\,B(w) + \half\,\EA(\half\,t)\,\pdd\,\EB(t,v)\,A\,v
   - \half\,\EA(\half\,t)\,A\,w - B(\EA(\half\,t)\,w)\,,
\end{align*}
which can be written in the form
\begin{subequations}
\label{eq:S1expr-herleitung}
\begin{align}
\S1(t,u)
&= \EA(\half\,t)\,B(w) - B(\EA(\half\,t)\,w) \label{eq:S1-line1} \\
& \quad {} + \half\,\EA(\half\,t)\big(\pdd\,\EB(t,v)\,A\,v - A\,\EB(t,v) \big)\,. \label{eq:S1-line2}
\end{align}
\end{subequations}
In order to find an integral representation for $ \S1(t,u) $, we separately consider the
terms~\eqref{eq:S1-line1} and~\eqref{eq:S1-line2}, with $ v $ and $ w $ fixed.
Differentiating with respect to $ t $ we find that they
satisfy the following evolution equations.
\begin{subequations}
\begin{description}
\item[\eqref{eq:S1-line1}:]
$ \S1_{(a)}(t) = \EA(\half\,t)\,B(w) - B(\EA(\half\,t)\,w) $ satisfies $ \S1_{(a)}(0)=0 $, and
\begin{equation} \label{eq:S1-line1-de}
\pdt{} \S1_{(a)}(t) = \half\,A\,\S1_{(a)}(t) +\half [A,B](\EA(\half\,t)\,w)\,.
\end{equation}
\end{description}
In~\eqref{eq:S1-line2}, we consider the expression within $ \big(\ldots \big) $:
\begin{description}
\item[\eqref{eq:S1-line2}:]
$ \S1_{(b)}(t) = \pdd\,\EB(t,v)\,A\,v - A\,\EB(t,v) $ satisfies $ \S1_{(b)}(0)=0 $, and
\begin{equation} \label{eq:S1-line2-de}
\pdt{} \S1_{(b)}(t) = B'(\EB(t,v))\,\S1_{(b)}(t,v) + [B,A](\EB(t,v))\,.
\end{equation}
\end{description}
\end{subequations}
Finally, applying~\eqref{eq:VOCLIN-1} and~\eqref{eq:VOCLIN-2}, recombination and
substituting $ v = \EA(\shalf t)\,u $,\, $ w = \EB(t,\EA(\shalf t)\,u) $
leads to the integral representation~\eqref{eq:S1integral} for $\S1(t,u) $.

\subsection{The second-order defect $ \S2(t,u) $}\label{sec:S2herleitung}
To evaluate $\S2$ defined in~\eqref{eq:S2def}, we proceed in an analogous way as for $ \S1 $, with $v$, $w$ defined in~\eqref{eq:wv_def}.\\
We start by differentiating the expression for $\S1$ from~\eqref{eq:S1expr}
with respect to $t$,
\begin{align*}
 \pdt{}\S1(t,u) &= \big(A + B'(\nS(t,u)) \big)\,\S1(t,u) \\
& \quad {} - B'(\EA(\half\,t)\,w)\,\EA(\half\,t)\,\pdd\,\EB(t,v)\,A\,v
           + \EA(\half\,t)\,B'(w)\,\pdd\,\EB(t,v)\,A\,v \\
& \quad {} - A\,\EA(\half\,t)\,B(w) + A\,B(\EA(\half\,t)\,w)
           - \half\,A\,\EA(\half\,t)\,\pdd\,\EB(t, v)\,A\,v\\
& \quad {} + \tfrac{1}{4}\,\EA(\half\,t)\,\pdd\,\EB(t,v)\,A^2\,v
           + \tfrac{1}{4}\,A^2\,\EA(\half\,t)\,w \\
& \quad {} + \tfrac{1}{4}\,\EA(\half\,t)\,\pdd^2\,\EB(t,v)(A\,v,A\,v)
           + \EA(\half\,t)\,B'(w)\,B(w) \\
& \quad {} + B'(\EA(\half\,t)\,w)\,B(\EA(\half\,t)\,w)
           - 2\,B'(\EA(\half\,t)\,w)\,\EA(\half\,t)\,B(w) \,.
\end{align*}
Subtracting $ F'(\nS(t,u))\,\S1(t,u) = (A+B'(\nS(t,u)))\,\S1(t,u)$ and rearranging terms yields
\begin{subequations}
\label{eq:S2}
\begin{align}
\S2(t,u)
&= \big(\EA(\half\,t)\,B'(w) - B'(\EA(\half\,t)\,w)\,\EA(\half\,t) \big)\,\pdd\,\EB(t,v)\,A\,v \label{eq:S2-line1} \\
& \quad {} + \big(A + B'(\EA(\half\,t)\,w)\big)\,\big(B(\EA(\half t)\,w) - \EA(\half t)\,B(w) \big) \label{eq:S2-line2} \\
& \quad {} + \big(\EA(\half\,t)\,B'(w) - B'(\EA(\half\,t)\,w)\,\EA(\half\,t) \big)\,B(w) \label{eq:S2-line3}\\
& \quad {} + \tfrac{1}{4}\,\EA(\half\,t)
            \Big(A \big(A\,\EB(t,v) - \pdd\,\EB(t,v)\,A\,v \big)        \label{eq:S2-line4} \\
& \qquad\qquad\qquad\quad {} - \big(A\,\pdd\,\EB(t,v) - \pdd\,\EB(t,v)\,A \big)\,A\,v \label{eq:S2-line5} \\
& \qquad\qquad\qquad\quad {} + \pdd^2\,\EB(t,v)(A\,v,A\,v)  \Big) \label{eq:S2-line6}\,.
\end{align}
\end{subequations}
Evaluation of $\S2(t,u)$ at $t=0$ shows $\S2(0,u)=0$, hence $\S2(t,u)=\Order(t)$.

In order to find an integral representation for $ \S2(t,u) $, we separately consider the
terms~\eqref{eq:S2-line1}--\eqref{eq:S2-line6},
with $ v $, $w$ and $ z :=\pdd\,\EB(t,v)\,A\,v $ fixed.
Differentiating with respect to $ t $ we find that they satisfy the following linear evolution equations.
\begin{subequations}
\begin{description}
\item[\eqref{eq:S2-line1}:]
$ \S2_{(a)}(t) = \big(\EA(\half\,t)\,B'(w) - B'(\EA(\half\,t)\,w)\,\EA(\half t) \big) z $
satisfies $ \S2_{(a)}(0)=0 $, and
\begin{align} \label{eq:S2-line1-de}
\pdt{} \S2_{(a)}(t)
&= \half\,A\,\S2_{(a)}(t) \\
& \quad {} + \half\,[A,B'(\EA(\half\,t)\,w)] A\,\EA(\half\,t) w + \half\,[A,B'(\EA(\half\,t)\,w)]\,\EA(\half\,t)\,\big(z -A\,w  \big) \notag \\
& \quad {} - \half\,B''(\EA(\half\,t)\,w)\,
                \big(A\,\EA(\half\,t)\,w,
                      \EA(\half\,t)\,(z- A\,w+ A\,w) \big)\,. \notag
\end{align}
\item[\eqref{eq:S2-line2}:]
$ \S2_{(b)}(t) = \big(A + B'(\EA(\half\,t)\,w)\big)\,\big(B(\EA(\half t)\,w) - \EA(\half t)\,B(w) \big) $
satisfies $ \S2_{(b)}(0)=0 $, and
\begin{align} \label{eq:S2-line2-de}
\pdt{} \S2_{(b)}(t)
&= \half\,A\,\S2_{(b)}(t) \\
& \quad {} + \half\,A\,[B\,,A]\,\EA(\half\,t)\,w + \half\,B'(\EA(\half\,t)\,w)\,[B\,,A]\,\EA(\half\,t)\,w \notag \\
& \quad {} - \half\,B''(\EA(\half\,t)\,w)\big(A\,\EA(\half\,t)\,w,\EA(\half\,t)\,B(w) -B(\EA(\half\,t)\,w) \big) \notag \\
& \quad {} + \half\,[A,B'(\EA(\half\,t)\,w)]\,\big(\EA(\half\,t)\,B(w) -  B(\EA(\half\,t)\,w)  \big)\,. \notag
\end{align}
\item[\eqref{eq:S2-line3}:]
$ \S2_{(c)}(t) = \big(\EA(\half\,t)\,B'(w) - B'(\EA(\half\,t)\,w)\,\EA(\half t) \big)\,B(w) $,
satisfies $ \S2_{(c)}(0)=0 $, and
\begin{align} \label{eq:S2-line3-de}
\pdt{} \S2_{(c)}(t)
&= \half\,A\,\S2_{(c)}(t) 
                       + \half\,[A,B'(\EA(\half\,t)\,w)]\,\EA(\half\,t)\,B(w) \\ 
& \qquad\qquad\qquad {} - \half\,B''(\EA(\half\,t)\,w)\big(A\,\EA(\half\,t)\,w,\EA(\half\,t)\,B(w) \big)\,. \notag
\end{align}
\end{description}
In~\eqref{eq:S2-line4}--\eqref{eq:S2-line6}, we consider the expressions within 
$ \big(\ldots \big) $:
\begin{description}
\item[\eqref{eq:S2-line4}:]
$ \S2_{(d)}(t) = A \big(A\,\EB(t,v) - \pdd\,\EB(t,v)\,A\,v \big) $
satisfies $ \S2_{(d)}(0)=0 $, and
\begin{align} \label{eq:S2-line4-de}
\pdt{} \S2_{(d)}(t)
&= B'(\EB(t,v))\big(\S2_{(d)}(t) \big) \\
& \quad {} + [A,[A,B]](\EB(t,v)) \notag \\
& \quad {} + 2\,[A,B'(\EB(t,v))]\,A\,\EB(t,v) - [A,B'(\EB(t,v))]\,\pdd\,\EB(t,v)\,A\,v \notag \\
& \quad {} - B''(\EB(t,v))\big(A\,\EB(t,v),A\,\EB(t,v) \big)\,. \notag
\end{align}
\item[\eqref{eq:S2-line5}:]
$ \S2_{(e)}(t) = \big(A\,\pdd\,\EB(t,v) - \pdd\,\EB(t,v)\,A \big)\,A\,v $
satisfies $ \S2_{(e)}(0)=0 $, and
\begin{equation} \label{eq:S2-line5-de}
\pdt{} \S2_{(e)}(t)
= B'(\EB(t,v))\big(\S2_{(e)}(t) \big) + [A,B'(\EB(t,v))]\,\pdd\,\EB(t,v)\,A\,v\,.
\end{equation}
\item[\eqref{eq:S2-line6}:]
$ \S2_{(f)}(t) = \pdd^2\,\EB(t,v)(A\,v,A\,v) $
satisfies $ \S2_{(f)}(0)=0 $, and
\begin{equation} \label{eq:S2-line6-de}
\pdt{} \S2_{(f)}(t)
= B'(\EB(t,v))\big(\S2_{(f)}(t) \big) + B''(\EB(t,v)) \big(\pdd\,\EB(t,v)\,A\,v,\pdd\,\EB(t,v)\,A\,v \big)\,.
\end{equation}
\end{description}
\end{subequations}
Finally, applying~\eqref{eq:VOCLIN-1} and~\eqref{eq:VOCLIN-2}, respectively, recombination and
substituting $ v = \EA(\shalf t)\,u $,\, $ w = \EB(t,\EA(\shalf t)\,u) $ and $z=\pdd\,\EB(t,v)\,A\,v$
leads to the integral representation~\eqref{eq:S2integral} for $\S2(t,u) $.

\section{Auxiliary estimates for the NLS case} \label{lab:gronwall}

\subsection{Estimate of $ \big\| \pdd\,\nE_F(\t_1 -\t_2,\nS(\t_2,u)) \cdot \S2(\t_2,u) \big\|_{L^2}$ } \label{C1}
For a detailed study of the estimate~\eqref{eq:pdEF}, we need to estimate
the arising expressions in $\L{2}(t,u)$ as
\begin{equation*}
\|\pdd\,\nE_F(t-\t_2,\nS(\t_2,u)) \S2(\t_2,u) \|_{L^2}
\leq C_1 +  C_2 \cdot \| \S2(\t_2,u)\|_{L^2}\leq C_1 +  C_2 \cdot \t_2\,\sup_{0 \leq \t_3 \leq \t_2}\| \s2(\t_2,\t_3,u)\|_{L^2} \,,
\end{equation*}
with constants $C_1,\,C_2$ resulting from Gronwall estimates.

We substitute
\begin{align*}
g=\S2(\t_2,u)\,, \quad w=\nS(\t_2,u)
\end{align*}
and apply the linear variation of constant formula in the following way,
\begin{align*}
\pdt{} \pdd\,\nE_F(t-\t_2,w)\,g
&= F'(\nE_F(t-\t_2,w)) \pdd\,\nE_F(t-\t_2,w)\,g \\
&= A\,\pdd\,\nE_F(t-\t_2,w)\,g + B'(\nE_F(t -\t_2,w))\,\pdd\,\nE_F(t-\t_2,w) g\,, \\
\pdd\,\nE_F(t-\t_2,w)\,\Big|_{t=\t_2} g &= g\,, \\
\Rightarrow \qquad
\pdd\,\nE_F(t-\t_2,w)\,g &= \EA(t-\t_2)\,g
+ \int_{\t_2}^{t} \EA(t-\theta)\,B'(\nE_F(\theta-\t_2,w))\,\pdd\,\nE_F(\theta-\t_2,w)\,g\,\dd\theta\,.
\end{align*}
Hence,
\begin{align*}
\| \pdd\,\nE_F(t-\t_2,w)\,g \|_{L^2}
 &\leq \| g \|_{L^2} + \int_{\t_2}^{t} \| B'(\nE_F(\theta-\t_2,w))\,
                                        \pdd\,\nE_F(\theta-\t_2,w)\,g \|_{L^2}\,\dd\theta \\
& \leq \| g \|_{L^2} + \int_{\t_2}^{t} \tfrac{1}{\eps}\,\| U\,\pdd\,\nE_F(\theta-\t_2,w)\,g \|_{L^2}\,\dd\theta\\
&\qquad   + \int_{\t_2}^{t} \tfrac{1}{\eps}\,\tilde{C}\,\th\,\|\nE_F(\theta-\t_2,w) \|_{H^2}^2 \|\,
                        \,\pdd\,\nE_F(\theta-\t_2,w)\,g \|_{L^2}\,\dd\theta\,.
\end{align*}
Applying a Gronwall argument we obtain
\begin{align*}
\| \pdd\,\nE_F(t-\t_2,w)\,g \|_{L^2} &\leq  \exp\Big(\tilde{C} \int_{\t_2}^{t} \tfrac{1}{\eps}\,|\th|\,\|\nE_F(\sig-\t_2,u) \|_{H^2}^2\, \dd\sig \Big)\\
& \qquad \quad {} \cdot \Big( \| g\|_{L^2} + \int_{\t_2}^{t} \tfrac{1}{\eps}\,\| U\,\pdd\,\nE_F(\theta-\t_2,w)\,g \|_{L^2}\,\dd\theta \Big)\,.
\end{align*}%
Now the question is how to argue a reasonable a priori estimate
for $ \tfrac{1}{\eps}\,\| U\,\pdd\,\nE_F(\theta-\t_2,w)\,g \|_{L^2}$.
In the following this is accomplished by relating this term to a known estimate for $\|U\, \EF(t,u)\|_{L^2}$, see~\cite{carles03}.

Considering the Fr{\'e}chet derivative of $\EF(\theta-\t_2,w+g)$ with a small increment $g$, ${\|g\|_{L^2}\leq \delta \| \EF(\theta-\t_2,w+g)\|}$ and $\| \EF(\theta-\t_2,w+g)\|_{L^2} = \| \EF(\theta-\t_2,w)\|_{L^2}+ \Order(\|g\|_{L^2})$,
\begin{align*}
\EF(\theta-\t_2,w+g) = \EF(\theta-\t_2,w)+\pdd\EF(\theta-\t_2,w)(g)+ \Order(\|g\|_{L^2}^2)\,,
\end{align*}
we obtain
\begin{align*}
U \cdot \big(\pdd\EF(\theta-\t_2,w)(g)\big) + \Order(\| g\|^2) & = U\cdot \EF(\theta-\t_2,w+g) - U \cdot \EF(\theta-\t_2,w)\,,\\
\|U\cdot\big(\pdd\EF(\theta-\t_2,w)(g)\big)\|_{L^2} & \leq \| U \cdot \EF(\theta-\t_2,w) - U \cdot \EF(\theta-\t_2,w+g)\|_{L^2} +\Order(\| g\|_{L^2}^2)\,,\\
\|U\cdot\big(\pdd\EF(\theta-\t_2,w)(g)\big)\|_{L^2} & \leq \| U \cdot \EF(\theta-\t_2,w)\|_{L^2} + \| U \cdot \EF(\theta-\t_2,w+g)\|_{L^2} + \Order(\|g\|_{L^2}^2)\,,
\end{align*}%
where the size of the increment $g=\S2(t,u)$ becomes negligible for sufficiently small choice of $t$.
Applying~\cite[pp. 532\,sqq.]{carles03} allows to bound $U \nE_F$ in $L^2$ by a constant $\frac{1}{2}\,C_{\ast}$, which depends on $\nE_F$. Altogether, we obtain the crude estimate
\begin{align}
\label{eq:CU-def}
\bcor{\sup_{\t_2\leq \theta\leq t}\| U \pdd\EF(\theta-\t_2,w)g\|_{L^2}\leq C_{\ast}} \,.
\end{align}%
Actually, the above derivation lets us expect that $C_\ast$ contains a factor $t$.
However, we have not been able to prove this in a rigorous way.

Altogether we obtain~\eqref{eq:pdEF},
\begin{equation}
\label{eq:Gronwall1}
\| \pdd\,\nE_F(t -\t_2,\nS(\t_2,u))\,g \|_{L^2}
\leq \exp\Big(\int_{\t_2}^{t} \tfrac{1}{\eps}\,\tilde{C}\,|\th|\,\|\nE_F(\sig-\t_2,\nS(\t_2,u))\|_{H^2}^2\,\dd \sig \Big)
\big( \| g\|_{L^2} + \tfrac{t}{\eps}\,C_\ast \big)\,.
\end{equation}
A similar result can be obtained in the $H^2$\,-\,norm,
\begin{equation}
\label{eq:GronwallH2}
\| \pdd\,\nE_F(t-\t_2,u)g\|_{H^2}
\leq \exp\Big(\tilde{C} \int_{\t_2}^{t} \tfrac{1}{\eps}\,|\th|\,\|\nE_F(\sig-\t_2,u) \|_{H^2}^2\, \dd\sig \Big)
\big( \| g\|_{H^2} + \tfrac{t}{\eps}\,C_\ast\big)\,,
\end{equation}
with a constant $\tilde{C}_\ast$ such that
\begin{align}
\label{eq:CUhat-def}
 \bcor{\sup_{\t_2\leq \theta\leq t}\| U \pdd\EF(\theta-\t_2,u)g\|_{H^2}\leq \tilde{C}_\ast}\,.
\end{align}
For simplicity of denotation, let \ccor{$ C_\ast $} be defined as the maximum of the constants
appearing in~(\ref{eq:CU-def}) and~\ref{eq:CUhat-def}).
In this sense the estimates from this section enter the local error estimates in Section~\ref{Lestim}.

\subsection{Estimate of $ \big\| \pdd^2 \nE_F(t-\t_2,\nS(\t_2,u))\big( (\S1(\t_2,u) \big)^2 \big\|_{L^2}$} \label{lab:gronwall-1}
For the estimate of $\| \pdd^2 \nE_F(t-\t_2,\nS(\t_2,u))\big(\S1(\t_2,u) \big)^2 \|_{L^2}$ in~\eqref{eq:pd2EF}, we proceed in a similar way as for~\eqref{eq:Gronwall1} with the help of the identity
\begin{align*}
\pdt{} \pdd^2\EF(t,u)\big(v,w\big) &= F''(\EF(t,u))\big(\pdd \EF(t,u)v,\pdd\EF(t,u)w\big) + F'(\EF(t,u))\pdd^2\EF(t,u)\big(v,w\big)\\
& =A\, \pdd^2\EF(t,u)\big(v,w\big) +  B''(\EF(t,u))\big(\pdd \EF(t,u)v,\pdd\EF(t,u)w\big)\\
& \quad  + B'(\EF(t,u))\pdd^2\EF(t,u)\big(v,w\big) \,,
\end{align*}
where
\begin{align*}
B''(u)(v,w) &= -2\,\ii\,\tfrac{1}{\eps}\th\,\big(\ol{u}\,v\,w + u\,\ol{v}\,w + u\,v\,\ol{w} \big)
\end{align*}
does not depend on $U$.

Again we can apply the variation of constant formula and obtain, with the help of Sobolev embeddings and~\eqref{eq:GronwallH2},
\begin{align*}
&\| \pdd^2 \nE_F (t-\t_2,\nS(\t_2,u)) (\S1(\t_2,u),\S1(\t_2,u))  \|_{L^2}   \\
& \quad {} \leq \exp \Big(C \int_{\t_2}^{t} \tfrac{1}{\eps}\,|\th|\,\|\nE_F(\sig-\t_2,\nS(\t_2,u)) \|_{H^2}^2\,\dd\sig \Big)\\
&\qquad \qquad {} \cdot \Big( \hat{C}\,\tfrac{t^3}{\eps}\,\| u \|_{L^2}
\big( \sup_{0 \leq \t_3 \leq \t_2} \| \s1(\t_2,\t_3,u) \|_{H^2}
+ \tfrac{1}{\eps}\,C_\ast \big)^2 + \tfrac{t^2}{\eps}\,C_\ast\Big)\,,
\end{align*}
for some constants $C$ and $\hat{C}$  depending on the Sobolev imbedding of $H^2$ in $L^2$.

\end{document}